\documentclass[12pt]{amsart}

\usepackage{stmaryrd}
\usepackage{comment}
\usepackage{color}
\usepackage{amsmath}    % \UTF{00E6}\UTF{0095}°\UTF{00E5}\UTF{00AD}\UTF{0160}\UTF{00E7}\UTF{0094}\UTF{0161}%
\usepackage{amssymb}  
\usepackage{amscd}
\usepackage{amsthm}  
\usepackage{graphicx}
\usepackage{txfonts} 
\usepackage[all]{xy} 	
\usepackage[dvipdfmx]{hyperref}%\UTF{00E5}\UTF{00AE}\UTF{009A}\UTF{00E7}\UTF{0178}\UTF{00A9}\UTF{00E3}\UTF{0081}\UTF{00AE}\UTF{00E3}\UTF{0083}\UTF{009E}\UTF{00E3}\UTF{0083}\UTF{0152}\UTF{00E3}\UTF{0082}\UTF{00AF}\UTF{00E3}\UTF{0081}\UTF{0161}\UTF{00E3}\UTF{0081}\UTF{008B}%
\numberwithin{equation}{section}
\theoremstyle{plain} 
\newtheorem{thm}{\indent\sc Theorem}[section] %\UTF{00E5}\UTF{00AE}\UTF{009A}\UTF{00E7}\UTF{0090}\UTF{0086}\UTF{00E3}\UTF{0081}\UTF{0161}\UTF{00E3}\UTF{0081}\UTF{008B}\UTF{00E3}\UTF{0081}\UTF{009D}\UTF{00E3}\UTF{0081}\UTF{00AE}\UTF{00E4}\UTF{00BB}\UTF{0096}\UTF{00E3}\UTF{0082}\UTF{0082}\UTF{00E3}\UTF{0082}\UTF{008D}\UTF{00E3}\UTF{0082}\UTF{0082}\UTF{00E3}\UTF{0082}\UTF{008D}%
\newtheorem{lem}[thm]{\indent\sc Lemma}
\newtheorem{cor}[thm]{\indent\sc Corollary} 
\newtheorem{prop}[thm]{\indent\sc Proposition}

\theoremstyle{definition}
\newtheorem{dfn}[thm]{\indent\sc Definition}
\newtheorem{exa}[thm]{\indent\sc Example}

\newtheorem{rem}[thm]{\indent\sc Remark}

\newcommand{\hol}[2] { #1 \rightarrow #2 }
\newcommand{\kah}{K\"{a}hler }

\newcommand{\Sym}[0]{\operatorname{Sym}}

\newcommand{\deldel}{\sqrt{-1}\partial \overline{\partial}}

\title[Characterization of pseudo-effective vector bundles by sHm]
{Characterization of pseudo-effective vector bundles \\ by singular Hermitian metrics}
\author{Masataka Iwai}

\address{Graduate School of Mathematical Sciences, The University of Tokyo, 3-8-1 Komaba,
Tokyo, 153-8914, Japan.}

 \email{{\tt
masataka@ms.u-tokyo.ac.jp, masataka.math@gmail.com}}

\subjclass[2010]{Primary 32J25, Secondary 14J60, 14E30}

\keywords
{Singular Hermitian metrics of vector bundle,
Pseudo-effective,
Weakly positive,
Big
}

\begin{document}
\maketitle

%%%%%%%%%%%%%%%%%%%%%%%%%%%%%%%%%%%%%%
\begin{abstract}
In this paper, we give complex geometric descriptions of the 
notions of algebraic geometric positivity of vector bundles and torsion-free coherent sheaves,
such as nef, big, pseudo-effective, and weakly positive, by using singular Hermitian metrics.
As an applications, we obtain a generalization of Mori's result.
%augumentedについても調べます.
We also give a characterization of  the augmented base locus %and restiricted base locus 
by using singular Hermitian metrics on vector bundles and the Lelong numbers.
%By using this results, we prove weakly positivity of direct images of pluri-adjoint line bundles with some assumption.
%Moreover we study a Lelong number of a singular Hermitian metric of a vector bundle defined by Berndtsson.
%We define the non \kah locus of a vector bundle, which is a higher rank analogy of Boucksom's 
%non \kah locus of a line bundle,
%and we prove the non \kah locus is equal to an augmented base locus of a vector bundle.

\end{abstract}
\tableofcontents
%\tableofcontents
%%%%%%%%%%%%%%%%%%%%%%%%%%%%%%%%%%%%%%%%%
\section{Introduction}

In \cite{Kodaira}, Kodaira proved that a line bundle $L$ is ample if and only if $L$ has a smooth Hermitian metric with positive curvature.
After that, Demailly \cite{Dem93} gave complex geometric descriptions of nef, big, and pseudo-effective line bundles.
For example, he proved that a line bundle $L$ is pseudo-effective if and only if $L$ has a singular Hermitian metric
with semipositive curvature current.
Ample, nef, big, and pseudo-effective are notions of algebraic geometric positivity.
Thus, their works related algebraic geometry to complex geometry.

The aim of this paper is to give complex geometric descriptions of notions of algebraic geometric positivity of vector bundles and torsion-free coherent sheaves.
Griffiths \cite{Gri} proved that if a vector bundle $E$ has a Griffiths positive metric, then $E$ is ample (i.e. $\mathcal{O}_{ \mathbb{P}(E) }(1)$ is ample).
The inverse implication is unknown. We do not know whether an ample vector bundle has a Griffiths positive metric.
This is so-called Griffiths' conjecture, which is one of longstanding open problems.
%In this section, we study the relationship between the metric positivity and the algebraic positivity.
In recent years, 
Liu, Sun, and Yang \cite{LSY} gave a partial answer to this conjecture.
\begin{thm}\cite[Theorem 1.2 and Corollary 4.6]{LSY}
Let $X$ be a smooth projective variety and $E$ be a holomorphic vector bundle on $X$.
If $E$ is ample, % (i.e. $\mathcal{O}_{ \mathbb{P}(E) }( 1)$ is ample  ) 
then there exists  $k \in \mathbb{N}_{>0}$ such that $\Sym^k(E)$ has a Griffiths $($Nakano$)$ positive smooth Hermitian metric.
\end{thm}
Throughout this paper, we will  denote by $\Sym^k(E)$ the $k$-th symmetric power of $E$ and 
denote by $ \mathbb{N}_{>0}$ the set of positive integers.
Inspired by the works of Liu, Sun, and Yang, we study notions of algebraic geometric positivity of vector bundles by using smooth and singular Hermitian metrics.
\begin{thm} 
%\footnote{When this paper was almost completed, the paper \cite{GG} and was submitted in ArXiv by Green and Griffiths.There exists an overlap of Theorem \ref{maintheorem} in \cite[Section 2 and Proposition 3.B.6]{GG}.}
\label{maintheorem}
Let $X$ be a smooth projective variety and $E$ be a holomorphic vector bundle on $X$.
%We fix a sufficiently ample line bundle $A$ on $X$.
\begin{enumerate}
\item $E$ is nef %(i.e. $\mathcal{O}_{ \mathbb{P}(E) }( 1)$ is nef) 
iff there exists an ample line bundle $A$ on $X$, such that $\Sym^k(E) \otimes A$ has a Griffiths semipositive smooth Hermitian metric for any $k \in \mathbb{N}_{>0}$.
\item $E$ is big \footnote{The definition of "big" in this paper is different from the definition of "big" in \cite{Laz2}.}
%この論文でのbigの定義はLazarsfeldとは違う.詳しくはdef参照のこと.}
%V-big (dd-ample at some point) 
iff there exist an ample line bundle $A$ 
and $k \in \mathbb{N}_{>0}$, such that $\Sym^{k}(E) \otimes A^{-1}$ has a Griffiths semipositive singular Hermitian metric.
\item $E$ is pseudo-effective %(weakly positive in the sense of Nakayama) 
iff there exists an ample line bundle $A$, such that 
$\Sym^{k}(E) \otimes A$ has a Griffiths semipositive singular Hermitian metric for any $k \in \mathbb{N}_{>0}$.
\item $E$ is weakly positive iff
there exist an ample line bundle $A$ and a proper Zariski closed set $Z$, such that 
 $\Sym^{k}(E) \otimes A$ has a Griffiths semipositive singular Hermitian metric $h_k$ for any $k \in \mathbb{N}_{>0}$
 and the Lelong number of $h_k$ at x is less than 2 for any $x \in X \setminus Z$. %and any $k \in \mathbb{N}_{>0}$.
 %$h_k$ is smooth outside $Z$.
 % and there exists a proper Zariski closed set $Z$ such that $h_k$ is smooth outside $Z$.
\end{enumerate}
\end{thm}
We will explain the definitions of big, pseudo-effective, %dd-ample at some point, 
and weakly positive in Section 2.
Further, we obtain similar results in the case of torsion-free coherent sheaves. 
We will discuss about torsion-free coherent sheaves in Section 5.
%Next, we study torsion-free coherent sheaves.
%Just like the case of vector bundles, we give complex geometric descriptions of
%big or weakly positive torsion-free coherent sheaves.

Nef, big, pseudo-effective, and weakly positive are notions of algebraic geometric positivity of vector bundles
and torsion-free coherent sheaves.
%which was investigated in \cite{Vie1}, \cite{Vie2} and \cite{Nak}.
In particular, Viehweg \cite{Vie1} proved that 
a direct image sheaf of an $m$-th relative canonical line bundle $f_{*}(m K_{X/Y})$ is weakly positive
for any fibration
$f \colon X \rightarrow Y$.
By using this result, he studied Iitaka's conjecture.
A Griffiths semipositive singular Hermitian metric, which is an analogy of a singular Hermitian metric of a line bundle and
a Griffiths semipositive metric,
was investigated in many papers. %(see \cite{BP},  \cite{Hos17}, \cite{HPS}, and so on). %\cite{DeC},\cite{Ina}, \cite{Iwa2}, and \cite{Rau} ).
%In particular, P\u{a}un and Takayama \cite{PT} proved that 
%$f_{*}(m K_{X/Y})$
%can be endowed with a Griffiths semipositive singular Hermitian metric
%for any fibration $f \colon X \rightarrow Y$.
%By using this result, 
By using Griffiths semipositive singular Hermitian metrics, 
Cao and P\u{a}un \cite{CP} proved Iitaka's conjecture 
when the base space is an Abelian variety. 
Therefore, our results %Theorem \ref{maintheorem} and Theorem \ref{torsionfreecoh} 
also relate algebraic geometry to complex geometry.
%By using this results, we generalize the Viehweg's weakly positivity theorem for
%$f_{*}(m K_{X/Y} \otimes L \otimes \mathcal{J}(h))$ with some assumption.

%Finally we also study the Lelong number of a singular Hermitian metric of a vector bundle, defined by Berndtsson \cite{Bernd2}.
%We define another Lelong number type invariant of a singular Hermitian metric of a vector bundle.
%Moreover, by using the Lelong number, we define the non \kah locus of a vector bundle, which is a higher rank analogy of the non \kah locus of a line bundle defined by Boucksom \cite{Bou}.
%We prove the non \kah locus of a vector bundle is equal to the augmented base locus of a vector bundle, defined in \cite{BKKMSU}.

%We have some applications about our results.
We have the following corollary.
\begin{cor}
\label{big-RC}
Let $X$ be a smooth projective $n$-dimensional variety.
If the tangent bundle $T_X$ is big, then $X$ is biholomorphic to $\mathbb{P}^n$.
\end{cor}

This corollary is a generalization of Mori's result: 
"If the tangent bundle $T_X$ is ample, then $X$ is biholomorphic to $\mathbb{P}^{n}$", since an ample vector bundle is big. %and $\mathbb{P}^n$ is rationally connected.
This corollary was proved by Fulger and Murayama \cite[Corollary 7.8]{FM}
 by using Seshadri constants of vector bundles .
We give an another proof by using singular Hermitian metrics.
Finally, we investigate the relationships between augmented base loci and Griffiths semipositive singular Hermitian metrics of vector bundles.
%また最後の章ではaugumented locusとsingular Hermitian metricの関係について調べる.
%The organization of this paper is as follows.
%In Section 2, we briefly recall the standard facts and results.
%In Section 3, we study a relationship between a metric of
 %$\mathcal{O}_{ \mathbb{P}(E) }( 1 ) $ and a metric of $E$
 %by using the methods and results of \cite{PT} and \cite{LSY}.
%In Section 4, we prove Theorem \ref{maintheorem}.
%By the results of Section 3, for a weakly positive vector bundle $E$,
%we show that $\Sym^k(E) \otimes A$ can be endowed with a Griffiths semipositive singular Hermitian metric.
%For the proof of the inverse implication, we use the methods of \cite{Iwa1}.
%As an application, we give Cutkosky-type criterion of V-big and pseudo-effective vector bundles.
%In Section 5, we study torsion-free coherent sheaves by the method in \cite{PT}.
%In Section 6, we define a new Lelong number of a singular Hermitian metric of a vector bundleand study relationship between this Lelong number and the Lelong number defined by Berndtsson \cite{Bernd2}.
%In Section 7, we define the non \kah locus of a vector bundle by using the new Lelong number.
%By using the methods of Section 4 and Hosono's singular Hermitian metric induced by global sections \cite{Hos17}, we prove a non \kah locus is equal to an augmented base locus.

{\bf Acknowledgment. } 
The author would like to thank his supervisor Prof. Shigeharu Takayama for helpful comments and enormous support.
He would like to thank Prof. Takayuki Koike for useful comments.
%He would like to thank Genki Hosono and Takahiro Inayama for useful comments.
He also would like to thank anonymous referee for some suggestions about the contents in Chapter \ref{augmented}.
%anonnymous refereeについてありがとうございます.
This work is supported by the Program for Leading Graduate Schools, MEXT, Japan. This work is also supported by JSPS KAKENHI Grant Number 17J04457.

\section{Definition and related results}

%\subsection{Singular Hermitian metrics}
Let $X$ be a smooth projective variety.
A function $\varphi$ : $\hol{X}{ [ -\infty , +\infty) } $ is said to
be {\it quasi-plurisubharmonic} %({\it quasi-plurisubharmonic}) 
if $\varphi$ is locally the sum of a 
plurisubharmonic function and a smooth function.
Let $L$ be a line bundle on $X$.
Fix a smooth metric $h_0$ on $L$.
$h$ is a {\it singular Hermitian metric} if $h = h_{0} e^{- \varphi}$ for some quasi-plurisubharmonic function $\varphi$.
%In addition, we say that $\varphi$ has { \it neat analytic singularities }
%if for any $x \in X$ there exists an open neighborhood $U$ of $x$ such that 
%$\varphi$ can be written
%$$
%\varphi = c \log \sum_{ 1 \le k \le L}| g_k(z) |^2 + w(z)
%$$
%on $U$ where $c\geq0$, $g_k \in \mathcal{O}_X(U)$ for any $k = 1 , \dots , L$ and $w \in \mathcal{C}^{\infty}(U)$.

%\begin{dfn}\cite[Chapter 5.]{Dem}
For any quasi-plurisubharmonic function $\varphi$ on $X$,
the{ \it multiplier ideal sheaf} $\mathcal{J}( e^{- \varphi}) $ is a coherent
subsheaf of $\mathcal{O}_X$ defined by
$$
\mathcal{J}( e^{- \varphi})_x \coloneqq
\{ f \in \mathcal{O}_{ X , x} ; \exists U \ni x , \int_{U} | f |^2 e^{ - \varphi} d \lambda < \infty \},
$$
where $U$ is an open coordinate neighborhood of $x$ and $d \lambda$ is the standard Lesbegue
measure in the corresponding open chart of $\mathbb{C}^n$.
The {\it Lelong number} $\nu(\varphi , x)$ at $x \in X$ is defined by
$$
\nu(\varphi , x) \coloneqq  \liminf_{z \rightarrow x} \frac{\varphi(z)}{ \log |z - x|}.
$$
%Let $(L , h)$ be a holomorphic line bundle with a singular Hermitian metric on $X$.
We define the {\it multiplier ideal sheaf} $\mathcal{J}( h )$ of a singular Hermitian metric $h$ on $L$  by
$
\mathcal{J}( h ) \coloneqq \mathcal{J}(  e^{ \log(h h_{0}^{-1})})
$, 
%where $h_{0}$ is a some smooth metric on $L$.
%\end{dfn}
%\begin{dfn}\cite[Chapter 2.]{Dem}
%For any quasi-plurisubharmonic function $\varphi$ on $X$ and a point $x \in X$,
%Let $(L , h)$ be a holomorphic line bundle with a singular Hermitian metric on $X$.
and define the {\it Lelong number} $\nu(h, x)$ of $h$ at $x\in X$ by 
$
\nu(h, x) \coloneqq \nu( -\log(h h_{0}^{-1}) , x).
$
%where $h_{0}$ is some smooth metric on $L$.
%\end{dfn}
We point out that 
both $\mathcal{J}( h ) $ and $\nu(h, x)$  do not depend
on the choice of $h_{0}$.

We review the definitions of singular Hermitian metrics of vector bundles.
%For more details, we refer the reader to \cite[Section 2]{Paun16}.
Let $H_r$ be the set of $r \times r$ semipositive definite Hermitian matrixs
and $\bar{H}_r$ be the space of semipositive, possibly unbounded Hermitian forms on
$\mathbb{C}^r$.

\begin{dfn} \cite[Definition 2.2.1 and 2.2.2]{PT} %\cite[Definiton 2.8 and Definition 2.9]{Paun16} \cite[Definition 2.2.1 and 2.2.2]{PT}
\begin{enumerate}
\item A {\it singular Hermitian metric} $h$ on $E$ is defined to be a locally measurable map with
values in $\bar{H}_r$ such that $0 < \det h < +\infty$ almost everywhere.
\item A singular Hermitian metric $h$ on $E$ is {\it Griffiths seminegative} 
({\it negatively curved})
if the function
$\log |u|^{2}_{h}$ is plurisubharmonic for any local section $u$ of $E$.
\item A singular Hermitian metric $h$ on $E$ is {\it Griffiths semipositive} 
({\it positively curved})
if the dual
singular Hermitian metric $h^{*} := {}^t\! h^{-1}$ on the dual vector bundle $E^{*}$ 
is Griffiths seminegative.
\end{enumerate}
\end{dfn}

These definitions are well-defined even if $E$ is a line bundle.
In particular, for any singular Hermitian metric $h$ on a line bundle $L$, 
$h$ is Griffiths semipositive iff $h$ has semipositive curvature current.
We recall the definition of a singular Hermitian metric on a torsion-free coherent sheaf.
Let $\mathcal{F} \neq 0$ be a torsion-free coherent  sheaf on $X$.
We will denote by $X_{\mathcal{F}}$ the maximal Zariski open set where $\mathcal{F}$ is locally free.
%We point out %that $\mathcal{F}|_{X_{\mathcal{F} } }$ is a vector bundle on $X_{\mathcal{F}}$ and$\text{codim}(X \setminus X_{ \mathcal{F} } ) \ge 2$.

\begin{dfn} \cite[Definition 2.4.1]{PT} %\cite[Definiton 2.8 and Definition 2.9]{Paun16} \cite[Definition 2.2.1 and 2.2.2]{PT}
\begin{enumerate}
\item A {\it singular Hermitian metric} $h$ on $\mathcal{F}$ is 
a singular Hermitian metric on the vector bundle $\mathcal{F}|_{X_{ \mathcal{F} }}$.
\item A singular Hermitian metric $h$ on $\mathcal{F}$ is {\it Griffiths seminegative} 
({\it negatively curved})
%(or  {\it h is seminegatively curved} )
if $h |_{X_{ \mathcal{F} } }$ is Griffiths seminegative.
\item A singular Hermitian metric $h$ on $\mathcal{F}$ is {\it Griffiths semipositive} 
({\it positively curved})
%(or {\it h is semipositively curved})
if there exists a Griffiths seminegative metric $g$ on $\mathcal{F}^{*}|_{X_{ \mathcal{F} }}$
such that $h |_{X_{ \mathcal{F} }} = (g |_{ X_{ \mathcal{F} } })^{*}$
\end{enumerate}
\end{dfn}
These are well-defined definitions (see \cite[Remark 2.4.2]{PT}). 
About a Griffiths semipositive singular Hermitian metric, 
P\u{a}un and Takayama proved the following Theorem.
%a direct image sheaf of
%a twisted relative canonical line bundle $f_{*} (K_{X/Y} \otimes L)$
%can be endowed with a Griffiths semipositive singular Hermitian metric
%for any fibration $f \colon X \rightarrow Y$.

\begin{thm}\cite[Theorem 1.1]{PT} \cite[Theorem 21.1 and Corollary 21.2]{HPS}
\label{pauntaka}
Let $f \colon X \rightarrow Y $ be a projective surjective morphism between connected complex manifolds
and $(L , h)$ be a holomorphic line bundle with a singular Hermitian metric with semipositive curvature current on $X$.
Then $f_{*}( K_{X/Y} \otimes L  \otimes \mathcal{J}(h))$ has a Griffiths semipositive singular Hermitian metric.

Moreover if the inclusion morphism
$$
f_{*}( K_{X/Y} \otimes L  \otimes \mathcal{J}(h)) \rightarrow  f_{*}( K_{X/Y} \otimes L )
$$
is generically isomorphism,
then 
$f_{*}( K_{X/Y} \otimes L  )$
also has a Griffiths semipositive singular Hermitian metric.
\end{thm}

%\subsection{Algebraic positivity of vector bundles}
%Throughout this paper, for any holomorphic vector bundle $E$, we will denoted by $\Sym^{k}(E)$ the $k$-th symmetric power of $E$.
%%%%%%%%%%%%%%%%%%%%%%%%%%%%%%%%
\begin{comment}
We review dd-ample and weakly positive.
\begin{dfn}\cite{Vie1} \cite{Vie2} \cite{Nak}\cite{BDPP}
Let $X$ be a smooth projective variety and $E$ be a holomorphic vector bundle on $X$.
\begin{enumerate}
\item $E$ is  {\it dd-ample at $x \in X$} ( "ample modulo double-duals" ) if there exist
$b \in \mathbb{N}_{>0} $ and an ample line bundle $A$ on $X$
such that $S^{b}(E) \otimes A^{  -1}$ is globally generated at $x$.
\item $E$ is {\it weakly positive at $x \in X$} if for any $a \in \mathbb{N}_{>0}$ and for any ample line bundle $A$, there exists 
$b \in \mathbb{N}_{>0}$ such that $\Sym^{ab}(E) \otimes A^{ b}$ is globally generated at $x$.
\item $E$ is {\it weakly positive in the sense of Nakayama} if there exists a point $x \in X$ such that $E$ is weakly positive at $x$.
\item $E$ is {\it weakly positive in the sense of Viehweg} if there exists a Zariski open set  $U \subset X$ such that $E$ is weakly positive at any $x \in U$.
\end{enumerate}
\end{dfn}
\end{comment}
%%%%%%%%%%%%%%%%%%%%%%%%%

Finally we recall some notions of algebraic positivity of torsion-free coherent sheaves.

\begin{dfn}\cite{Vie1} \cite{Vie2} \cite{Nak}\cite{BDPP}
\label{definition_psef}
Let $X$ be a smooth projective variety and $\mathcal{F}$ be a torsion-free coherent sheaf on $X$.
\begin{enumerate}
%\item $\mathcal{F}$ is  %{\it dd-ample at $x \in X$} ( "ample modulo double-duals" ) if there exist$b \in \mathbb{N}_{>0} $ and an ample line bundle $A$ on $X$ such that $\widehat{\Sym}^{b}( \mathcal{F}  ) \otimes A^{  -1}$ is globally generated at $x$.
\item $\mathcal{F}$ is {\it weakly positive at $x \in X$} if for any $a \in \mathbb{N}_{>0}$ and any ample line bundle $A$, there exists 
$b \in \mathbb{N}_{>0}$ such that $\widehat{\Sym}^{ab}( \mathcal{F}  ) \otimes A^{ b}$ is globally generated at $x$.
\item $\mathcal{F}$ is {\it pseudo-effective}
({\it weakly positive in the sense of Nakayama})
 if there exists $x \in X$ such that $ \mathcal{F}  $ is weakly positive at $x$.
\item $\mathcal{F}$ is {\it weakly positive} 
({\it weakly positive in the sense of Viehweg})
if there exists a Zariski open set  $U \subset X$ such that $ \mathcal{F}  $ is weakly positive at any $x \in U$.
\item $\mathcal{F}$ is {\it big} 
({\it big in the sense of Viehweg})
if there exist an ample line bundle $A$ and $a \in \mathbb{N}_{>0}$ such that $\widehat{\Sym}^{a}( \mathcal{F}  ) \otimes A^{ -1}$ is pseudo-effective. 

\end{enumerate}
\end{dfn}
We will  denote by $\widehat{\Sym}^{k}( \mathcal{F}  )$ 
the double dual of the $k$-th symmetric power of $ \mathcal{F} $.
%\begin{rem}
%\begin{enumerate}
%\item If $E$ is dd-ample at some point, then there exists a Zariski open set $U$ such that $E $ is dd-ample at any $x \in U$.
%\item If $E$ is weakly positive in the sense of Nakayama, then there exists a family $\{ Z_i \}_{i \in \mathbb{N}_{>0}}$ of proper Zariski closed sets such that $E$ is weakly positive at any $x \in X \setminus \cup_{i} Z_i$.
%\item
%\item %From $\mathbb{B}_{-}(E) \subset \overline{\mathbb{B}_{-}(E)} \subset \mathbb{B}_{+}(E)$, 
These definitions are well-defined even if $\mathcal{F}$ is a line bundle.
\begin{lem}
Let $X$ be a smooth projective variety and $\mathcal{F}$ be a torsion-free coherent sheaf on $X$.
If $\mathcal{F}$ is big, then $\mathcal{F}$ is weakly positive.
\end{lem}
\begin{proof}
By the assumption, there exist $k \in \mathbb{N}_{>0}$ and an ample line bundle $A$ such that
$\widehat{\Sym}^{k}( \mathcal{F}  ) \otimes A^{ -1}$ is pseudo-effective.
We may assume that $A$ is very ample.
Then we can take $l \in \mathbb{N}_{>0}$ and $x \in X$
such that $\widehat{\Sym}^{kl}( \mathcal{F}  ) \otimes A^{1 -l}$
 is globally generated at $x$.
 Hence there exists a Zariski open set  $U \subset X$
 such that $\widehat{\Sym}^{kl}( \mathcal{F}  ) $ is globally generated on $U$.
 Therefore, for any $a \in \mathbb{N}_{>0}$, 
$\widehat{\Sym}^{akl}( \mathcal{F}  ) \otimes A$
is globally generated on $U$.
%あるample line bundle $A$があって$\widehat{\Sym}^{k}( \mathcal{F}  ) \otimes A^{ -1}$ is pseudo-effective. である$A$をvery smpleと仮定して良い.よってある性のの整数$l$があって$\widehat{\Sym}^{kl}( \mathcal{F}  ) \otimes A^{-kl +l}$がglobally generatedである.これよりあるZariski open set  $U \subset X$があって$\widehat{\Sym}^{kl}( \mathcal{F}  ) $は$U$上でglobally generatedである.以上より任意の整数$a$について$\widehat{\Sym}^{akl}( \mathcal{F}  ) \otimes A^{ a}$は$U$上でglobally generatedである.
\end{proof}

For any vector bundle $E$, we obtain the following implication:

\begin{equation*}
\xymatrix@C=40pt@R=30pt{
  \txt{ample} \ar@{=>}[r]\ar@{=>}[d]& \txt{nef} \ar@{=>}[d] &
 % \txt{\textit{E} has a Griffiths semipositive \\ singular Hermitian metric} \ar@{=>}[d] 
 \\ 
 \txt{big} \ar@{=>}[r] &  \txt{weakly positive} \ar@{=>}[r]& \txt{pseudo-effective.}
}
\end{equation*}
%About notions of algebraic positivity, 

%$$\mbox{big}\Rightarrow\mbox{weakly positive}\Rightarrow\mbox{pseudo-effective}.$$
%We point out pseudo-effective is called weakly positive in the sense of Nakayama in \cite{PT}.
%\end{enumerate}
%\end{rem}

\section{A singular Hermitian metric on $\mathcal{O}_{ \mathbb{P}(E) }( 1 ) $ }

Throughout this paper, we will denote by $\pi \colon \mathbb{P}(E) \rightarrow X$ the projective bundle of rank one quotients of $E$,
and denote by $ \mathcal{O}_{ \mathbb{P}(E) }(1) $ the universal quotient of $\pi^{*} E$ on $\mathbb{P}(E)$.
We study a singular Hermitian metric on $\mathcal{O}_{ \mathbb{P}(E) }( 1 ) $
induced by a singular Hermitian metric on $E$.
\begin{lem}% \cite[Proposition 2.3.5]{PT}
\label{PTprop}
Let $X$ be a smooth projective $n$-dimensional variety,
$E$ be a holomorphic vector bundle of rank $r$ on $X$,
and $A$ be a line bundle on $X$.
Assume that there exists $m \in \mathbb{N}_{>0}$ such that 
$\Sym^m(E) \otimes A$ has a Griffiths semipositive singular Hermitian metric $h_m$.
Then $\mathcal{O}_{ \mathbb{P}(E) }(m) \otimes \pi^{*} A$ has a singular Hermitian metric $g_m$ with semipositive curvature current.

Moreover, for any $x \in X$, there exist an open set $V$ near $x$ and a positive constant $C_{V}$, such that 
$g_{m} \le C_{V} \pi^{*}(\det h_{m})$ on $\pi^{-1}(V)$.

\end{lem}

\begin{proof}
%\textcolor{red}{Use XyPic !}
We will denote by 
$\pi_{m} \colon  \mathbb{P}( \Sym^m(E)\otimes A  ) \rightarrow X$.
Notice that $ \mathbb{P}( \Sym^m(E) )  = \mathbb{P}( \Sym^m(E) \otimes A ) $ and 
$ \mathcal{O}_{ \mathbb{P}( \Sym^m(E) ) }(1 ) \otimes \pi_{m}^{*}(A) =  \mathcal{O}_{ \mathbb{P}( \Sym^m(E) \otimes A ) }(1)$.
Let $ \mu_{m} \colon \mathbb{P}( E )   \rightarrow \mathbb{P}( \Sym^m(E) )  $ be a standard 
$m$-th Veronese embedding.
Then $ \pi = \pi_{m} \circ \mu_{m}$ and 
$$
\mathcal{O}_{ \mathbb{P}(E) }(m) \otimes \pi^{*} A
  = 
 \mu_{m}^{*} (\mathcal{O}_{ \mathbb{P}( \Sym^m(E) ) }(1 )) \otimes \pi^{*}(A)
 =
 \mu_{m}^{*} ( \mathcal{O}_{ \mathbb{P}( \Sym^m(E) \otimes A ) }(1) ) .
$$

By \cite[Proposition 2.3.5]{PT}, 
$\mathcal{O}_{ \mathbb{P}( \Sym^m(E) \otimes A ) }(1)$ admits a singular Hermitian metric $\widetilde{g_m}$ with semipositive curvature current. %via $\pi_{m}^{*} ( \Sym^m(E) \otimes A ) \rightarrow \mathcal{O}_{ \mathbb{P}( \Sym^m(E) \otimes A ) }(1)$.
Therefore $g_m \coloneqq \pi_{m}^{*} \widetilde{g_m}$ is a singular Hermitian metric with semipositive curvature current on
$ \mathcal{O}_{ \mathbb{P}(E) }(m) \otimes \pi^{*} A$.

Fix $x \in X$.
Since $\pi_{m}^{-1}(x)$ is compact, 
there exist an open set $V$ near $x$ and a positive constant $C_{V}$ such that 
$\widetilde{ g_{m}} \le C_{V} \pi_{m}^{*}(\det h_{m})$ on $\pi^{-1}_{m}(V)$
by \cite[Proposition 2.3.5]{PT}.
Therefore, we have $g_{m} \le C_{V} \pi^{*}(\det h_{m})$ on $\pi^{-1}(V)$.
\end{proof}

\begin{lem}
\label{analytic}
Let $X$ be a smooth projective $n$-dimensional variety, $E$ be a holomorphic vector bundle of rank $r$ on $X$, and 
$A$ be a line bundle on $X$.
Assume that there exist $m \in \mathbb{N}_{>0}$ and $x \in X$ such that $\Sym^{m}(E) \otimes A$ is globally generated at $x$.
Then there exist a singular Hermitian metric $g$ with semipositive curvature current on $\mathcal{O}_{ \mathbb{P}(E) }(m) \otimes \pi^{*} A$ and a proper Zariski 
closed set $Z \subset X$, such that $g$ is smooth outside $\pi^{-1} (Z)$.

Moreover if there exists a Zariski open set $U \subset X$ such that  $\Sym^{m}(E) \otimes A$ is globally generated %at $x$ for any $x \in U$, 
on $U$, we can take $Z$ such that $ Z \cap U = \varnothing$.
\end{lem}

\begin{proof}
Let $\{ U_i \}$ be a finite open cover of $X$ such that $U_i$ is a coordinate neighborhood
and $\pi^{-1} (U_i)$ is biholomorphic to  $ U_i \times \mathbb{P}^{r-1}$.
We take a local holomorphic frame $e_{1}, \dots ,e_{r}$ of $E$ on $U_i$
and a local holomorphic frame $e_A$ of $A$ on $U_i$.
Let $s_1, \dots ,s_N$ be a basis on $H^0( X ,  \Sym^m (E) \otimes A)$.
Set $M \coloneqq \binom{m + r -1}{r}$ and  
$$s_{j} = \displaystyle 
\sum_{\alpha } 
f_{j \alpha} e_{1}^{\alpha_1} \cdots e_{r}^{\alpha_r} e_A ,
$$
where $f_{j\alpha}$ are holomorphic function on $U_{i}$ and the sum
is taken over $\alpha = (\alpha_1 , \cdots ,\alpha_r) \in \mathbb{N}_{>0}^{r}$ with $\alpha_1 + \cdots + \alpha_r = m$.
The $N \times M$ matrix $B^{(i)} $ is defined by
$
B^{(i)} 
= 
(f_{j \alpha}).
$
Set $Z_{i} \coloneqq \{ z \in U_i \colon \text{rank } B^{(i)}(z) < M \} $ and $Z \coloneqq \cup Z_i$.
Since $\Sym^{m}(E) \otimes A$ is globally generated at $x$, we have $N \ge M$ and $Z $ is a proper Zariski closed set of $X$.
We define the singular Hermitian metric $g$ with semipositive curvature current on $\mathcal{O}_{ \mathbb{P}(E) }(m) \otimes \pi^{*} A$,
induced by the global sections $\pi^{*}(s_1), \cdots, \pi^{*}(s_N)  \in H^{0}( \mathbb{P}(E) , \mathcal{O}_{ \mathbb{P}(E) }(m) \otimes \pi^{*} A )$ (see \cite[Example 3.14]{Dem}).

We show that $g$ is smooth outside $\pi^{-1}(Z)$.
We will denote by $e_{1}^{*}, \dots ,e_{r}^{*}$ the dual frame on $E^{*}$.
The corresponding holomorphic coordinate on $E^{*}$ are denoted by $(W_1, \cdots, W_r)$.
We may regard $\pi^{-1} (U_i)$ as $ U_i \times \mathbb{P}^{r-1}$.
We take the chart $\{ [W_1 : \cdots :W_r ] \in \mathbb{P}^{r-1} \colon W_r \neq  0\}$.
We will define the isomorphism by 
$$
\begin{array}{ccc}
U_i \times \{ W_r \neq  0\}   & \rightarrow & U_i \times \mathbb{C}^{r-1} \\
(z , [W_1 : \cdots : W_r])  & \rightarrow  & (z , \frac{W_1}{W_r} , \cdots , \frac{W_{r-1}}{W_r})
\end{array}
$$
and we regard $U_i \times \{ W_r \neq  0\} $  as $ U_i \times \mathbb{C}^{r-1} $.
Put $\eta_{l} \coloneqq \frac{W_{l}}{W_r} $ for $1 \le l \le r-1$ and $\eta_{r} \coloneqq 1$.
In this setting, we have
$$
 \mathcal{O}_{ \mathbb{P}(E) }(-1) |_{ U_i \times \mathbb{C}^{r-1} } = 
 \{ (z ,  \eta , \xi) \in U_i \times \mathbb{C}^{r-1} \times \mathbb{C}^r
 \colon
 \eta_{i} \xi_{j} = \eta_{j} \xi_{i} \}
$$
and the local section 
$$
e_{ \mathcal{O}_{ \mathbb{P}(E) }(-1) } ( z , (\eta_1 , \cdots , \eta_{r-1 } ))
=
(z , (\eta_1 , \cdots, \eta_{r-1 } ) ,(\eta_1 , \cdots , \eta_{r-1 }  , 1) ).
$$

The local section $e_{ \mathcal{O}_{ \mathbb{P}(E) }(1)  }  $ of $\mathbb{P}(E) (1)$ is defined by the dual of $e_{ \mathcal{O}_{ \mathbb{P}(E) }(-1)  }  $.
Then we have
$$
\pi^{*}(s_{j}) |_{ U_i \times \mathbb{C}^{r-1} } =\sum_{\alpha} f_{j \alpha} (z)\eta_{1}^{\alpha_1} \cdots \eta_{r-1}^{\alpha_{r-1}} 1^{\alpha_{r}} e_{ \mathcal{O}_{ \mathbb{P}(E) }(1)  } ^{m} \pi^{*}(e_A)
$$
by using the isomorphism 
$H^{0}( X , \Sym^m(E) \otimes A) \simeq H^{0}(\mathbb{P}(E) ,  \mathcal{O}_{ \mathbb{P}(E) }(m) \otimes \pi^{*}(A) )$.

Since $g$ is defined by $1 / ( \sum_{1 \le j \le N} | \pi^{*}(s_{j})|^2  )$,
$g$ is described on $U_i \times \mathbb{C}^{r-1}$
by
$$
g(z,\eta_{1}, \cdots, \eta_{r-1}) = 
\Bigl( \sum_{1 \le j \le N}|\sum_{\alpha} f_{j \alpha} (z)\eta_{1}^{\alpha_1} \cdots \eta_{r-1}^{\alpha_{r-1}} 1^{\alpha_{r}} |^2 \Bigr)^{-1}.
$$
Therefore it is enough to show that $g^{-1} (z , \eta_{1} , \cdots, \eta_{r-1}) \neq 0$ for any
$(z , \eta_{1}, \cdots, \eta_{r-1}) \in (U_i \setminus W) \times \mathbb{C}^{r-1}$.
It is easily to check by the definition of $Z$ and the standard linear algebra.

The second statement is also easily proved by the definition of $Z$.
\end{proof}

\begin{cor}
\label{keycor}
Let $X$ be a smooth projective $n$-dimensional variety, $E$ be a holomorphic vector bundle of rank $r$ on $X$, and 
$A$ be a line bundle on $X$.
Assume that there exist $m , b \in \mathbb{N}_{>0}$ and $x \in X$ such that 
$\Sym^{(m+r)b}(E) \otimes (A \otimes \det E^{*})^{b}$ is globally generated at $x$.
Then there exist a Griffiths semipositive singular Hermitian metric $h$ on $\Sym^{m }(E) \otimes A $ and a proper Zariski 
closed set $Z \subset X$ such that $h$ is smooth outside $Z$.

Moreover if there exists a Zariski open set $U \subset X$ such that $\Sym^{(m+r)b}(E) \otimes (A \otimes \det E^{*})^{b}$ is globally generated on $U$, %at $x$ for any $x \in U$, 
we can take $Z$ such that $Z \cap U = \varnothing$.
\end{cor}

\begin{proof}
By Lemma \ref{analytic} and dividing by $b$,
there exist a singular Hermitian metric $g$ with semipositive curvature current on 
$\mathcal{O}_{ \mathbb{P}(E) }(m + r) \otimes \pi^{*} ( A \otimes \det E^{*} )$ and a proper Zariski 
closed set $Z \subset X$ such that $g$ is smooth outside $\pi^{-1} (Z)$.
%($\mathcal{O}_{ \mathbb{P}(E) }(m + r) \otimes \pi^{*} ( A \otimes \det E^{*} )$ is a line bundle !).
From $\det E \simeq \pi_{*}( K_{\mathbb{P}(E) / X} \otimes  \mathcal{O}_{ \mathbb{P}(E) }(r) )$,
we have 
$$
\Sym^{m }(E) \otimes A
\simeq
\pi_{*}( K_{\mathbb{P}(E) / X} \otimes  \mathcal{O}_{ \mathbb{P}(E) }( m + r) \otimes \pi^{*} ( A \otimes \det E^{*} ) )
$$
and the inclusion morphism
$$
\pi_{*}( K_{\mathbb{P}(E) / X} \otimes  \mathcal{O}_{ \mathbb{P}(E) }( m +r ) \otimes \pi^{*} ( A \otimes \det E^{*} ) \otimes \mathcal{J}(g))
\rightarrow
\pi_{*}( K_{\mathbb{P}(E) / X} \otimes  \mathcal{O}_{ \mathbb{P}(E) }( m +r) \otimes \pi^{*} ( A \otimes \det E^{*} ) )
$$
is generically isomorphism.
By Theorem \ref{pauntaka}, 
$\Sym^m(E) \otimes A$ has a 
Griffiths semipositive singular Hermitian metric $h$ such that $h$ is smooth outside $Z$ (see \cite[Chapter 22]{HPS}).

The proof of the second statement is same as above.
\end{proof}

\section{Proof of Theorem \ref{maintheorem} and Corollary \ref{big-RC} }
In this section, we prove Theorem \ref{maintheorem}.
First, we study a pseudo-effective vector bundle .
\begin{thm}
\label{weaklypositive}
Let $X$ be a smooth projective $n$-dimensional variety and $E$ be a holomorphic vector bundle of rank $r$ on $X$.
The followings are equivalent.
\begin{enumerate}
\item[(A)] $E$ is pseudo-effective. %(weakly positive in the sense of Nakayama).

%\item There exists an ample line bundle $A$ and a point $x \in X$ such that, 
%for any positive integer $a$, there exists a positive integer
%$b$ such that $\Sym^{ab}(E) \otimes A^{ b}$ is globally generated at $x$.

\item[(B)]  There exists an ample line bundle $A$ such that $\Sym^{k}(E) \otimes A$ has 
a Griffiths semipositive singular Hermitian metric $h_k$ for any $k \in \mathbb{N}_{>0}$.
Moreover, for any $k \in \mathbb{N}_{>0}$, there exists a proper Zariski closed set $Z_k \subset X$ 
such that $h_k$ is smooth outside $Z_k$.

%\item For any ample line bundle $A$, there exists an positive integer $c$ such that $\Sym^{k}(E) \otimes A^c$ has 
%a positively curved singular Hermitian metric $h_k$ for any positive integer $k$.
%Moreover, for any $k$, there exists a proper Zariski closed set $Z_k \subset X$ 
%such that $h_k$ is smooth outside $Z_k$.

\item[(C)]  There exists an ample line bundle $A$ such that $\Sym^{k}(E) \otimes A$ has 
a Griffiths semipositive singular Hermitian metric $h_k$ for any $k \in \mathbb{N}_{>0}$.

%\item For any ample line bundle $A$, there exists an positive integer $c$  such that $\Sym^{k}(E) \otimes A^c$ has 
%a positively curved singular Hermitian metric for any positive integer $k$.
\end{enumerate}
Moreover if $E$ satisfies the condition $($C$)$, then $E$ is weakly positive at any
$x \in X \setminus \cup_{k \in \mathbb{N}_{>0}} \{ z \in X \colon \nu( \det h_{k} , z ) \ge 2) \} $.
\end{thm}

\begin{proof}
(A) $\Rightarrow$ (B). 
We take $x \in X$ such that $E$ is weakly positive at $x$, and take an ample line bundle $A$ such that $A \otimes \det E^{*}$ is ample.
For any $a \in \mathbb{N}_{>0}$, there exists $b \in \mathbb{N}_{>0}$ such that
$\Sym^{(a + r)b}(E) \otimes ( A \otimes \det E^{*} )^{b}$ is globally generated at $x$.
By Corollary \ref{keycor}, the proof is complete.

(B) $\Rightarrow$ (C). Clear.

(C) $\Rightarrow$ (A).
%We can prove by using the method of \cite[Proposition 2.3.5 and Lemma 2.3.6]{PT}. However for the completeness, we give a proof in another way.
The proof will be divided into 3 steps.

Step 1. %{\bf Preliminary}
Fix an ample line bundle $H$.
%It is enough to show that for any  $a $, there exist a point $x \in X$ and a positive integer $b$ such that $\Sym^{ab}(E) \otimes H^{b}$ is globally generated at $x$.
By Siu's Theorem \cite[Corollary 13.3]{Dem},
the set $Z_{k} \coloneqq  \{ z \in X \colon \nu( \det h_{k} , z ) \ge 2) \} $ is a proper Zariski closed set.
Fix $x \in X \setminus \cup_{k} Z_k $.
We take a local coordinate $(U ; z_1 , \cdots , z_n)$ near $x$.
Let $ \varphi = \eta(n+1) \log |z-x|^2$, where $\eta $ is a cut-off function such that $\eta \equiv 1 $ near $x$.
%and we put $\psi \coloneqq \frac{n}{n+1} \pi^{*} \phi$.
Let $h_{H}$ be a positive smooth Hermitian metric on $H$.
We take $b \in \mathbb{N}_{>0}$ such that 
\begin{enumerate}
\item $A^{-1} \otimes K_{X} ^{-1} \otimes \det E^{*} \otimes H^b$ is ample, and 
%\item There exists a Zariski open set $U$ such that $\epsilon( H^b , z ) > n+1$ for any $z \in U$.
\item $b\sqrt{-1}\Theta_{H , g_{H} } +  \deldel \varphi \ge 0$ in the sense of current.
\end{enumerate}
From $\Sym^{2ab}(E) \otimes H^{2b} \simeq \pi_{*} ( \mathcal{O}_{ \mathbb{P}(E) }(2ab) \otimes \pi^{*} H^{2b} )$, it is enough to show that the restriction map
$$
H^0(  \mathbb{P}(E) , \mathcal{O}_{ \mathbb{P}(E) }(2ab) \otimes \pi^{*} H^{2b} )
\rightarrow
H^0( \pi^{-1}(x) , \mathcal{O}_{ \mathbb{P}(E) }(2ab) \otimes \pi^{*} H^{2b} | _{\pi^{-1}(x)})
$$
is surjective for any $a \in \mathbb{N}_{>0}$.

Step 2. %{\bf Taking a singular Hermitian metric}
From $\pi^{*} (\det E) \simeq  K_{\mathbb{P}(E) / X} \otimes  \mathcal{O}_{ \mathbb{P}(E) }(r) $,
we have
$$
\mathcal{O}_{ \mathbb{P}(E) }(2ab) \otimes \pi^{*} H^{2b} 
\simeq
K_{\mathbb{P}(E)} \otimes ( \mathcal{O}_{ \mathbb{P}(E) }(2ab + r) \otimes \pi^{*}A )
 \otimes \pi^{*} ( A^{-1} \otimes K_{X} ^{-1} \otimes \det E^{*} \otimes H^b) \otimes \pi^{*}(H^b).
$$

Since $\Sym^{2ab + r}(E) \otimes A$ has 
a Griffiths semipositive singular Hermitian metric $h_{2ab + r}$,
by Lemma \ref{PTprop},
$\mathcal{O}_{ \mathbb{P}(E) }(2ab + r) \otimes \pi^{*}A$ has a singular Hermitian metric $g_{2ab + r}$
with semipositive curvature current.
By Skoda's theorem \cite[Lemma 5.6]{Dem} and Lemma \ref{PTprop},
 there exist an open set $x \in V \subset X$ and a positive constant $C$ such that
 \begin{enumerate}
\item $g_{2ab + r} \le C \pi^{*}(\det h_{2ab+r})$ holds on $\pi^{-1}(V)$,
\item $\det h_{2ab+r} \in L^1(V)$, and
\item $\varphi =(n+1) \log |z-x|^2 $ holds on $V$.

\end{enumerate}
Since $A^{-1} \otimes K_{X} ^{-1} \otimes \det E^{*} \otimes H^b$ is ample,
there exists a smooth positive metric $g_{1}$ on $A^{-1} \otimes K_{X} ^{-1} \otimes \det E^{*} \otimes H^b$.
Set $\widetilde{L} \coloneqq  ( \mathcal{O}_{ \mathbb{P}(E) }(2ab + r) \otimes \pi^{*}A )
 \otimes \pi^{*} ( A^{-1} \otimes K_{X} ^{-1} \otimes \det E^{*} \otimes H^b) \otimes \pi^{*}(H^b)$, 
$\widetilde{g} \coloneqq g_{2ab + r}\pi^{*} (g_{1} g_{H}^b)$,
and $\psi \coloneqq \frac{n}{n+1} \pi^{*} \varphi$.
Then the followings hold.
\begin{enumerate}
\item $ K_{\mathbb{P}(E)} \otimes \widetilde{L} \simeq \mathcal{O}_{ \mathbb{P}(E) }(2ab) \otimes \pi^{*} H^{2b} $.
\item $\widetilde{g} $ is a singular Hermitian metric with semipositive curvature current on $\widetilde{L}$.
\item For any $\alpha \in [0 , 1]$, we have
$\sqrt{-1}\Theta_{\widetilde{L} , \widetilde{g}  } + (1 + \frac{\alpha }{n} ) \deldel \psi \ge 0$
in the sense of current.
\end{enumerate}

Step 3. %{\bf Global extension by an $L^2$ estimate}
Fix a \kah form $\omega_{ \mathbb{P}(E)  }$ on $\mathbb{P}(E) $. % and a \kah form $\omega_{X}$ on $X$.
If necessarily we take $V$ small enough, we may assume that $\pi^{-1} (V)$ is biholomorphic on $V \times \mathbb{P}^{r-1}$.
Therefore, there exists $s_{V} \in H^0( \pi^{-1}(V), K_{\mathbb{P}(E)} \otimes \widetilde{L} )$
such that $s_{V} | _{\pi^{-1}(x)} = s$.
%We show that $s_V$ is extended globally.
We take a cut-off function $\rho$ on $V$ such that $\rho = 1$ near $x$ and $\inf_{ \text{supp}( \overline{\partial} \rho)} \varphi > - \infty$.
%\begin{enumerate}\item  $\rho = 1$ near $x$, and\item $\inf_{ supp( \overline{\partial} \rho)} \varphi > - \infty$.\end{enumerate}
%where $dV_{X , \omega_{X}} \coloneqq \frac{ \omega^ {n } }{n!}$ is a volume form on $X$.
Set $\widetilde{\rho} \coloneqq \pi^{*} \rho$.
We solve the global $\overline{\partial}$-equation $\overline{\partial} F = \overline{\partial}( \widetilde{\rho} s_V)$
with the weight $\widetilde{g} e^{-\psi}$.

First, we have
$$
\|  \widetilde{\rho} s_V \|^{2}_{ \widetilde{g} , \omega_{ \mathbb{P}(E)  } }
= 
\int_{ \pi^{-1}(V) } | \widetilde{\rho} s_V |^{2}_{ \widetilde{g} , \omega_{ \mathbb{P}(E)  } }dV_{\omega_{ \mathbb{P}(E)  } , \mathbb{P}(E)   }
\le
C_{1} \int_{ \pi^{-1}( V) } | \pi^{*}\det  h| dV_{\omega_{ \mathbb{P}(E) }, \mathbb{P}(E)   }
<
+ \infty,
$$
 where $C_1$ is a positive constant.
Similarly, it is easy to check 
$\|  \overline{\partial}( \widetilde{\rho} s_V) \|^{2}_{ \widetilde{g} , \omega_{ \mathbb{P}(E)  } } < +\infty$.
Therefore $\overline{\partial}( \widetilde{\rho} s_V)$ gives rise to a cohomology class 
$[\overline{\partial}( \widetilde{\rho} s_V)]$ which is $[\overline{\partial}( \widetilde{\rho} s_V)] = 0$
in $H^{1}(\mathbb{P}(E) , K_{\mathbb{P}(E)} \otimes \widetilde{L} \otimes \mathcal{J}( \widetilde{g}))$.

Second, we have
\begin{align*}
\| \overline{\partial}( \widetilde{\rho} s_V)\|^{2}_{ \widetilde{g} e^{-\psi} , \omega_{ \mathbb{P}(E)  } }
=
\int_{ \pi^{-1}(V) } | \overline{\partial}( \widetilde{\rho} s_V) |^{2}_{ \widetilde{g} e^{-\psi} , \omega_{ \mathbb{P}(E)  } }dV_{\omega_{ \mathbb{P}(E)  } , \mathbb{P}(E)   }
\le
C_{2}
\int_{ \pi^{-1}( supp (\overline{\partial} \rho)) } | \pi^{*} (\det h) | e^{- \psi} dV_{\omega_{ \mathbb{P}(E)  } , \mathbb{P}(E)   } 
&<
+ \infty ,
\end{align*}
where $C_2$ is a positive constant.
Therefore
$\overline{\partial}( \widetilde{\rho} s_V)$ is a $\overline{\partial}$-closed $( n + r -1 , 1 )$ form
with $\widetilde{L}$ value which is square integrable the weight of $\widetilde{g} e^{-\psi}$.

By the injectivity theorem in \cite[Theorem 1.5]{CDM},
the natural morphism
$$
H^{1}(\mathbb{P}(E) , K_{\mathbb{P}(E)} \otimes \widetilde{L} \otimes \mathcal{J}( \widetilde{g} e^{-\psi}))
\rightarrow
H^{1}(\mathbb{P}(E) , K_{\mathbb{P}(E)} \otimes \widetilde{L} \otimes \mathcal{J}( \widetilde{g}))
$$
is injective.
From $[\overline{\partial}( \widetilde{\rho} s_V)] = 0$ in $H^{1}(\mathbb{P}(E) , K_{\mathbb{P}(E)} \otimes \widetilde{L} \otimes \mathcal{J}( \widetilde{g}))$,
we have
$[\overline{\partial}( \widetilde{\rho} s_V)] = 0$ in $H^{1}(\mathbb{P}(E) , K_{\mathbb{P}(E)} \otimes \widetilde{L} \otimes \mathcal{J}( \widetilde{g} e^{-\psi} ))$.
Hence we obtain an $( n + r -1 , 1 )$ form
$F$ with $\widetilde{L}$ value which is square integrable with the weight $\widetilde{g} e^{-\psi} $
such that $\overline{\partial}F = \overline{\partial}( \widetilde{\rho} s_V)$.

We will show that $F |_{ \pi^{-1} (x)} \equiv 0$.
To obtain a contradiction, we assume that $F(z) \neq 0$ for a point $z \in  \pi^{-1} (x)$.
We take an open set $x \in W \subset \subset V$, an open set $W' \subset \mathbb{P}^{r-1}$, and a positive constant $C_3$ such that
$W \times W' \subset \subset \pi^{-1}(V)$ and $|F|^{2}_{ \widetilde{g} } \ge C_3$ on $W$.
Thus we have
\begin{align*}
%+ \infty
%>
\| F \|^{2}_{ \widetilde{g} e^{-\psi} , \omega_{ \mathbb{P}(E)  } }
=
\int_{ \mathbb{P}(E) } | F|^{2}_{ \widetilde{g} e^{-\psi} , \omega_{ \mathbb{P}(E)  } }dV_{\omega_{ \mathbb{P}(E)  } , \mathbb{P}(E)   }
&\ge
\int_{ W \times W'  } | F |^{2}_{ \widetilde{g} e^{-\psi} , \omega_{ \mathbb{P}(E)  } }dV_{\omega_{ \mathbb{P}(E)  } , \mathbb{P}(E)   }\\
&\ge
C_{4}
\int_{ W \times W'   } e^{- \psi}dV_{\omega_{ \mathbb{P}(E)  } , \mathbb{P}(E)   } \\
&\ge
C_5 \int_{ W \times W'   } e^{- n  \log |z - x|^2  }dV_{\omega_{ \mathbb{P}(E)  } , \mathbb{P}(E)   }
\\
&=
+ \infty ,
\end{align*}
where $C_4$ and $C_5$ are positive constants.
This is a contradiction from $ \| F \|^{2}_{ \widetilde{g} e^{-\psi} , \omega_{ \mathbb{P}(E)  } } < + \infty$.

Therefore we put  $S \coloneqq \widetilde{\rho} s_V - F \in H^0( \mathbb{P}(E) , K_{\mathbb{P}(E)} \otimes \widetilde{L} ) $, then 
$S | _{\pi^{-1}(x)} = ( \widetilde{\rho}  s_V - F) | _{\pi^{-1}(x)} = s $, which completes the proof.
Hence $E$ is weakly positive at any $x \in X \setminus \cup_{k} Z_k $.
\end{proof}

By the same argument, we have the following corollary.
\begin{cor}
\label{weakviehweg}
Let $X$ be a smooth projective $n$-dimensional variety and $E$ be a holomorphic vector bundle of rank $r$ on $X$.
The followings are equivalent.
\begin{enumerate}
\item[(A)] $E$ is weakly positive. %in the sense of Viehweg.
\item[(B)]  There exist an ample line bundle $A$ and a proper Zariski closed set $Z \subset X$, such that $\Sym^{k}(E) \otimes A$ has 
a Griffiths semipositive singular Hermitian metric $h_k$ for any $k \in \mathbb{N}_{>0}$
and $h_k$ is smooth outside $Z$.

\item[(C)]  There exist an ample line bundle $A$ and a proper Zariski closed set $Z \subset X$, such that $\Sym^{k}(E) \otimes A$ has 
a Griffiths semipositive singular Hermitian metric $h_k$ for any  $k \in \mathbb{N}_{>0}$
and
 $\cup_{k} \{ z \in X \colon \nu( \det h_k , z ) \ge  2\} \subset Z$.

\end{enumerate}

\end{cor}
\begin{proof}
(B) $\Rightarrow$ (C) is clear. By Theorem \ref{weaklypositive}, we obtain (C) $\Rightarrow$ (A).
We give a proof of (A) $\Rightarrow$ (B).
By the definition, there exists a Zariski open set $U$ such that $E$ is a weakly positive at any $x \in U$.
We take an ample line bundle $A$ such that $A \otimes \det E^{*}$ is ample.
Fix $a \in \mathbb{N}_{>0}$.
For any $m \in \mathbb{N}_{>0}$, we define $Z_m$ by the Zariski closed set of points $x \in X$ such that $ \Sym^{(a + r)m}E \otimes (A \otimes \det E^{*})^{m}$ is not globally generated at $x$.
%then we have 
%$$
%\bigcap_{m} {\rm Bs}(S^{(a + r)m}E \otimes (A \otimes \det E^{*})^{m}) \subset X \setminus U.
%$$
Then we obtain $b \in \mathbb{N}_{>0}$ such that $Z_b =  \cap_{ m \in \mathbb{N}_{>0}} Z_m$.
%$$
%{\rm Bs}(S^{(a+ r)b}E \otimes (A \otimes \det E^{*})^{b}) = \bigcap_{m} {\rm Bs}(S^{(a + r)m}E \otimes (A \otimes \det E^{*})^{m}).
%$$
Thus, $\Sym^{(a + r)b}E \otimes (A \otimes \det E^{*})^{b}$ is globally generated at any $x \in  U$ by $Z_b \subset X \setminus U$.
By Corollary \ref{keycor}, $ \Sym^{a}E \otimes A$ has a Griffiths semipositive singular Hermitian metric $h$ which is smooth on  $U$, the proof is complete.
\end{proof}
%We put $Z\coloneqq \cap_{ m \in \mathbb{N}_{>0}} Z_m$ and we have $Z \subset X \setminus U$.
%We point out if $E$ satisfies the condition (C), then $E$ is weakly positive at any $x \in X \setminus Z$.
The following corollary was already proved in \cite{PT}.
We give an another proof.

\begin{cor} \cite[Proposition 2.3.5]{PT}
\label{sHmweak}
Let $X$ be a smooth projective variety and $E$ be a holomorphic vector bundle on $X$.
If $E$ has a Griffiths semipositive singular Hermitian metric $h$,
then $E$ is weakly positive at any $x \in  \{ z \in X \colon \nu( \det h , z ) = 0\}  $.
In particular, $E$ is pseudo-effective.
\end{cor}
\begin{proof}
Since $E$ has a Griffiths semipositive singular Hermitian metric, 
$\Sym^{k}(E)$ also has a Griffiths semipositive singular Hermitian metric $\Sym^k(h)$ for any $k \in \mathbb{N}_{>0}$.
Therefore, for any ample line bundle $A$, $\Sym^{k}(E) \otimes A$ has a Griffiths semipositive singular Hermitian metric $\Sym^k(h) h_{A}$, where $h_A$ is a smooth metric with positive curvature on $A$.
Since we have 
$$\cup_{k \in \mathbb{N}_{>0}} \{ z \in X \colon \nu( \det \Sym^k(h) h_{A} , z ) \ge 2) \} 
= \{ z \in X \colon \nu( \det h , z ) > 0\},
$$
$E$ is weakly positive at any $x \in \{ z \in X \colon \nu( \det h , z ) = 0\}$ by Theorem \ref{weaklypositive}.
\end{proof}

Next, we treat big vector bundles.
\begin{cor}
\label{Vbig}
Let $X$ be a smooth projective $n$-dimensional variety and $E$ be a holomorphic vector bundle of rank $r$ on $X$.
The followings are equivalent.
\begin{enumerate}
\item[(A)] $E$ is big.

\item[(B)]  There exist $k \in \mathbb{N}_{>0}$, an ample line bundle $A$, and 
a proper Zariski closed set $Z \subset X$, such that $\Sym^{k}(E) \otimes A^{-1}$ has 
a Griffiths semipositive singular Hermitian metric $h$ 
and $h$ is smooth outside $Z$.

\item[(C)]
There exist an ample line bundle $A$ and $k \in \mathbb{N}_{>0}$, such that $\Sym^{k}(E) \otimes A^{-1}$ has 
a Griffiths semipositive singular Hermitian metric.

\end{enumerate}
\end{cor}

\begin{proof}
(A) $\Rightarrow$ (B).
%By Theorem \ref{bigness}, 
There exist an ample line bundle $A$
and $b \in \mathbb{N}_{>0}$, such that $\Sym^{b}(E) \otimes A^{-1}$ is pseudo-effective.
By Theorem \ref{weaklypositive}, 
there exists an ample line bundle $H$ such that 
$\Sym^{kb}(E) \otimes A^{-k} \otimes H$ 
has a Griffiths semipositive singular Hermitian metric $h_{k}$
for any $k \in \mathbb{N}_{>0}$.
Moreover, there exists a proper Zariski closed set $Z_k$ such that $h_k$ is smooth outside $Z_k$.
Therefore we take $k \in \mathbb{N}_{>0}$ such that $A^{k} \otimes H^{-1}$ is ample, which completes the proof.

(B) $\Rightarrow$ (C). Clear.

(C) $\Rightarrow$ (A).
For any $a \in \mathbb{N}_{>0}$, $\Sym^{a}( \Sym^k( E ) \otimes  A^{-1} ) \otimes A$ has 
a Griffiths semipositive singular Hermitian metric.
By Theorem \ref{weaklypositive}, $\Sym^k( E ) \otimes  A^{-1} $ is pseudo-effective, which completes the proof.
%Therefore, $E$ is V-big by Theorem \ref{bigness}.
\end{proof}

{\it Proof of Corollary \ref{big-RC} }. 
%By \cite[Corollary 5.4]{Pet06}, it is enough to find a curve $C$ such that $T_{X} |_{C}$ is ample.
By Corollary \ref{Vbig}, there exist  $k \in \mathbb{N}_{>0}$, an ample line bundle $A$, and a proper Zariski closed set $Z \subset X$, such that $\Sym^{k}(T_{X}) \otimes A^{-1}$ has 
a Griffiths semipositive singular Hermitian metric $h$ which is smooth on $X \setminus Z$.
%and $h$ is smooth outside $Z$.
%We take curve $C = D_{1} \cap \dots \cap D_{n-1}$ for some $D_{i} \in | m H |$ and $m >> 0 $
%such that 
%$h|_{C}$ is a singular Hermitian metric on $\Sym^{k}(T_{X}) \otimes A^{-1} |_{C}$.
Since $X$ is uniruled, it is enough to show that $K_{X}^{-1}.C \ge n+1$ for any $x \in X \setminus Z$ and for any rational curve $C$ through $x$ by \cite[Cor 0.4]{CMSB02}.

Fix $x \in X \setminus Z$ and  a rational curve $C$ through $x$ .
First we will show that $T_{X}|_{C}$ is ample.
By \cite[Theorem 6.4.15]{Laz}, it is enough to show that any quotient bundle of $T_{X}|_{C}$ has positive degree.
Fix a quotient bundle $G$ of $T_{X} |_{C}$ and a smooth positive metric $h_A$ on $A$.
$\Sym^k G$ has a quotient metric $h_{\Sym^k G}$ induced by $ (h h_A) |_{C}$ on $\Sym^{k}(T_{X} |_C) $. %which is a Griffiths positive singular Hermitian metric.
%Therefore $\det h_{G}$ is singular Hermitian metric with positive curvature current on $\det G$.
Therefore $\det G$ has a singular Hermitian metric $h_{\det G}$ with positive curvature current
by some root of $\det h_{\Sym^k G}$.
We have
$$
\deg G
= \int_{C} c_{1}(G)
=\int_{C} c_{1}(\det G,h_{\det G})
= \int_{C}\frac{\sqrt{-1}}{2 \pi }\Theta_{\det G ,h_{\det G}}
> 0, 
$$
thus $T_{X}|_{C}$ is ample. 

Since $C$ is a rational curve,  we obtain
$$
T_{X}|_{C} \cong \mathcal{O}_{C}(a_1) \oplus \cdots \oplus \mathcal{O}_{C}(a_n) , 
$$ 
where $a_i$ is integer for any $1 \le i \le n$ , $a_1 \ge a_2 \ge \dots \ge a_n$, and $a_1 \ge 2$.
Since $T_{X}|_{C}$ is ample, we have $a_n \ge 1$, therefore  $K_{X}^{-1}.C = a_1 + \cdots + a_n \ge n+1$, which completes the proof.
%\begin{proof}
%By \cite[Theorem 1.2]{iwai}, there exists a positive integer $m \in \mathbb{N}_{>0}$ such that $Sym^{m}(T_X)$ has a Griffiths positive singular Hermitian metric. 
%Therefore by the above argument, 
%we find a curve $C$ such that $T_{X} |_{C}$ is ample, which completes the proof.

%\end{proof}

Finally, we study a nef vector bundle.
\begin{prop}[cf. \cite{DPS1} Theorem 1.12 ] %[cf. \cite {DPS1} Theorem 1.12 ]
Let $X$ be a smooth projective variety and $E$ be a holomorphic vector bundle of rank $r$ on $X$.
$E$ is nef $($i.e. $\mathcal{O}_{ \mathbb{P}(E) }( 1)$ is nef$)$ iff there exists an ample line bundle $A$ on $X$ 
such that $\Sym^k(E) \otimes A$ has a Griffiths positive smooth Hermitian metric for any  $k \in \mathbb{N}_{>0}$.
\end{prop}
\begin{proof}
($\Rightarrow$)
We assume that $E$ is nef.
We take an ample line bundle $H$ on $X$ such that $H \otimes \det E^{*}$ is ample,
%By \cite[Section 2.7]{Harts}, 
and take $N \in \mathbb{N}_{>0}$ such that $E \otimes (H \otimes \det E^{*})^N$
is ample. 
%that is, $\mathcal{O}_{ \mathbb{P}(E) }( 1 ) \otimes \pi^{*} ( H \otimes \det E^{*} )^N$ is ample.
For any $k \in \mathbb{N}_{>0}$, we have
$$
\Sym^k(E) \otimes H \otimes (H \otimes \det E^{*})^{N-1}
\simeq
\pi_{*}(  K_{\mathbb{P}(E) / X} \otimes  \mathcal{O}_{ \mathbb{P}(E) }( k +r) \otimes \pi^{*} ( H \otimes \det E^{*} )^N ).
$$
Since $\mathcal{O}_{ \mathbb{P}(E) }( 1) $ is nef,
$ \mathcal{O}_{ \mathbb{P}(E) }( k +r) \otimes \pi^{*} ( H \otimes \det E^{*} )^N $ is ample.
Therefore, $\Sym^k(E) \otimes H \otimes (H \otimes \det E^{*})^{N-1} $ has a Griffiths semipositive smooth Hermitian metric for any $k \in \mathbb{N}_{>0}$.
We put $A \coloneqq H^2 \otimes (H \otimes \det E^{*})^{N-1}$, the proof is complete.

($\Leftarrow$)
%We will denoted by 
%$\pi_{m} \colon  \mathbb{P}( \Sym^m(E)\otimes A  ) \rightarrow X$.
%We have $ \mathbb{P}( \Sym^m(E) )  = \mathbb{P}( \Sym^m(E) \otimes A ) $ and 
%$ \mathcal{O}_{ \mathbb{P}( \Sym^m(E) ) }(1 ) \otimes \pi_{m}^{*}(A) =  \mathcal{O}_{ \mathbb{P}( \Sym^m(E) \otimes A ) }(1)$.
Let $ \mu_{k} \colon \mathbb{P}( E )   \rightarrow \mathbb{P}( \Sym^k(E) )  = \mathbb{P}( \Sym^k(E) \otimes A ) $ be a standard 
$k$-th Veronese embedding.
%Then we have $ \pi = \pi_{m} \circ \mu_{m}$ and 
Since 
$\mathcal{O}_{ \mathbb{P}( \Sym^k(E) \otimes A ) }(1) $ is ample and
$\mathcal{O}_{ \mathbb{P}(E) }(k) \otimes \pi^{*} A
  = 
 \mu_{k}^{*} ( \mathcal{O}_{ \mathbb{P}( \Sym^k(E) \otimes A ) }(1) )$,
$\mathcal{O}_{ \mathbb{P}(E) }(k) \otimes \pi^{*} A$ is ample  for any $k \in \mathbb{N}_{>0}$.
Therefore $\mathcal{O}_{ \mathbb{P}(E) }(1)$ is nef.
\end{proof}
%Hence we finish the proof of Theorem \ref{maintheorem}. 

\begin{exa}[Cutkosky's criterion]
Let $X$ be a smooth projective variety and $L_1, \dots , L_r$ be holomorphic line bundles on $X$.
The vector bundle $E$ is defined by $E \coloneqq \oplus_{i = 1} ^{r} L_i$.
By \cite[Chapter 2.3.B]{Laz},
we have the following criterions.
\begin{enumerate}
\item $E $ is ample iff any $L_i$ is ample.
\item $E $ is nef iff any $L_i$ is nef.
\end{enumerate}
We give a generalization of Cutkosky's criterion of big and pseudo-effective.
\begin{lem}
\begin{enumerate}
\item $E $ is big iff any $L_i$ is big.
\item $E $ is pseudo-effective iff any $L_i$ is pseudo-effective.
Moreover $E $ is pseudo-effective iff $E$ has a Griffiths semipositive singular Hermitian metric.
\end{enumerate}
\end{lem}
\begin{proof}
(1) ($\Rightarrow$)
If $E$ is big, then there exist an ample line bundle $A$, $c \in \mathbb{N}_{>0}$, and a Zariski open set $U$, such that $\Sym^{c}(E) \otimes A^{-1}$ is globally generated on $U$.
Thus $L_{i}^{ c} \otimes A^{-1}$ is globally generated on $U$ for any $1 \le i \le r$.
%and consequently we have $\mathbb{B}_{+}(L_i) \subset X \setminus U$. 
Therefore $L_i$ is big.

(1) ($\Leftarrow$)
Let $A $ be an ample line bundle and $h_A$ be a smooth metric with positive curvature
on $A$, such that $\omega = \sqrt{-1}\Theta_{A , h_A}$ is a \kah form on $X$.
Since $L_i$ is big, there exist a singular Hermitian metric $h_i$ and positive number $\epsilon_{i}$
such that $\sqrt{-1}\Theta_{ L_i, h_i} \ge \epsilon_{i} \omega$.
We define a singular Hermitian metric $h$ on $E$ by $h = \oplus_{i=1}^{r} h_{i}$.
We take $c \in \mathbb{N}_{>0}$ such that $\min_{1 \le i \le r}\epsilon_i > 2 / c$.
Then $\Sym^{c}(E) \otimes A^{-1}$ has a Griffiths semipositive singular Hermitian metric $\Sym^{c}(h) h_{A}^{-1}$,
which completes the proof.

(2) ($\Rightarrow$)
Fix $x \in X$ such that $E$ is weakly positive at $x$.
For any ample line bundle $A$ and $a \in \mathbb{N}_{>0}$, there exists $b \in \mathbb{N}_{>0}$ such that $\Sym^{ab}(E) \otimes A^{ b} $ is globally generated at $x$.
Hence
$L_{i}^{ab} \otimes A^{b}$ is globally generated at $x$ for any $1 \le i \le r$.
 %we have $x \notin \mathbb{B}_{-}(L_i) $ for any $1 \le i \le r$. 
Therefore $L_i$ is pseudo-effective.

(2) ($\Leftarrow$)
Since $L_i$ is pseudo-effective, $L_i$ has a singular Hermitian metric $h_i$ 
with semipositive curvature current.
Hence $h := \oplus_{i = 1}^{r} h_i$ is a Griffiths semipositive singular Hermitian metric on $E$.
Therefore $E$ is pseudo-effective by Corollary \ref{sHmweak}.
\end{proof}

\end{exa}

\begin{rem}
In general, a pseudo-effective vector bundle does not necessarily have a Griffiths semipositive singular Hermitian metric.
By Hosono \cite[Example 5.4]{Hos17}, there exists a nef vector bundle which does not have a Griffiths semipositive singular Hermitian metric.
%Therefore, weakly positive in the sense of Nakayama (Viehweg) does not always imply theexistence of a Griffiths semipositive singular Hermitian metric.$E_H$ is obtained as follows.

Let $C$ be an elliptic curve.
$E$ is defined by the nontrivial exact sequence of vector bundles:
$$
0 \rightarrow  \mathcal{O}_{C} \rightarrow E \rightarrow  \mathcal{O}_{C} \rightarrow 0.
$$
By \cite[Theorem 1.4]{HIM}, if $E$ has a Griffiths semipositive singular Hermitian metric, then the above exact sequence splits, which is impossible. 
%We have $E_H \simeq E_{H}^{*}$ and $E_H$ is nef by \cite[Example 1.7]{DPS1}.Hosono proved that if $E_H$ has a Griffiths seminegative singular Hermitian metric $h_{H}$,$h_H$ is locally described by $$h_H=\begin{pmatrix}0 & 0 \\0 & b\end{pmatrix},$$where $b$ is some real number.Since a singular Hermitian metric $h$ satisfies the condition $0 < \det h < + \infty$ almost everywhere,this is a contradiction.Hence $E_H$ does not have a Griffiths semipositive (seminegative) singular Hermitian metric.
\end{rem}
%Therefore, $E_H$ does not have a Griffiths semipositive singular Hermitian metric.

%%%%%%%%%%%%%%%%%%%%%

%\section{Kawamata-Viehweg vanishing theorem}
\section{On the case of torsion-free coherent sheaves}

%Throughout this paper, we will denote by $\Sym^k( \mathcal{F} )$ the $k$-th symmetric power of $\mathcal{F}$ and denote by $\widehat{\Sym}^{k}( \mathcal{F} )$ the double dual of the sheaf of $\Sym^k( \mathcal{F} )$.
\begin{thm}
\label{torsionfreecoh}
Let $X$ be a smooth projective variety and $\mathcal{F} \neq 0$ be a torsion-free coherent  sheaf on $X$.
\begin{enumerate}
\item $\mathcal{F}$ is pseudo-effective iff
there exists an ample line bundle $A$, such that $\widehat{\Sym}^k( \mathcal{F} ) \otimes A$ has 
a Griffiths semipositive singular Hermitian metric for any $k \in \mathbb{N}_{>0}$.
\item $\mathcal{F}$ is weakly positive iff
there exist an ample line bundle $A$ and a Zariski open set $U \subset X$, such that $\widehat{\Sym}^k( \mathcal{F} ) \otimes A$ has 
a Griffiths semipositive singular Hermitian metric $h_k$ %for any  $k \in \mathbb{N}_{>0}$
and
the Lelong number of $h_k$ at x is less than 2 for any $x \in U$ and any $k \in \mathbb{N}_{>0}$.
\item $\mathcal{F}$ is big  iff
there exist an ample line bundle $A$ and $k \in \mathbb{N}_{>0}$, such that $\widehat{\Sym}^k( \mathcal{F} ) \otimes A^{-1}$ has 
a Griffiths semipositive singular Hermitian metric.
\end{enumerate}
\end{thm}
%Now, we prove Theorem \ref{torsionfreecoh}.
\begin{proof}
We put  $E \coloneqq \mathcal{F}|_{X_{\mathcal{F} } }$, which is a vector bundle on $X_{\mathcal{F}}$.
Since $\widehat{\Sym}^{k}( \mathcal{F} ) \otimes A$ is reflexive for any $k \in \mathbb{N}_{>0}$,
%by \cite[Chapter 5]{Har94}, 
we have 
\begin{equation}
\label{reflex}
H^{0}(X_{\mathcal{F} } , \Sym^k( E ) \otimes A )
%= H^{0}(X_{\mathcal{F} } , \Sym^k( \mathcal{F} ) \otimes A )
%\simeq H^{0}(X_{\mathcal{F} } , \widehat{\Sym}^{k}( \mathcal{F} ) \otimes A )
\simeq
H^{0}(X , \widehat{\Sym}^{k}( \mathcal{F} ) \otimes A ).
\end{equation}

(1)($\Rightarrow$).
We assume that $\mathcal{F}$ is pseudo-effective.
We take $x \in X_{\mathcal{F}}$ such that $\mathcal{F}$ is weakly positive at $x$, and take an ample line bundle $A$ such that $A \otimes ( \widehat { \det \mathcal{F} } )^{*}$ is ample.
For any $k \in \mathbb{N}_{>0}$, there exists $b \in \mathbb{N}_{>0}$ such that
$\widehat{\Sym}^{kb}( \mathcal{F}) \otimes ( A \otimes ( \widehat { \det \mathcal{F} } )^{*} )^{b}$ is globally generated at $x$.
Therefore by Equation \ref{reflex}, 
%the vector bundle 
$\Sym^{kb}(E) \otimes ( A \otimes \det E^{*} )^{b}$ on $X_{\mathcal{F}}$ is globally generated at $x$.
By the argument of Corollary \ref{keycor},
there exists a Griffiths semipositive singular Hermitian metric $h$ on 
 $(\widehat{\Sym}^{k}( \mathcal{F}) \otimes A ) |_{ X_{\mathcal{F}} }= \Sym^k(E) \otimes A$  
($h$ is smooth outside a countable union of proper Zariski closed sets).
From $\text{codim}(X \setminus X_{ \mathcal{F} } ) \ge 2$,
$h$ extends to $X_{(\widehat{\Sym}^{k}( \mathcal{F}) \otimes A )}$. %by \cite[Remark 2.18]{Paun16}.
Therefore $\widehat{\Sym}^{k}( \mathcal{F}) \otimes A $ has a Griffiths semipositive singular Hermitian metric $h$.

(1)($\Leftarrow$).
From $X_{ \mathcal{F} }  \subset X_{(\widehat{\Sym}^{k}( \mathcal{F}) \otimes A )} $ for any
$k \in \mathbb{N}_{>0}$,
%the vector bundle 
$\Sym^k(E) \otimes A$ has a Griffiths semipositive singular 
Hermitian metric for any
$k \in \mathbb{N}_{>0}$.
By using the argument of the proof of (C) $\Rightarrow$ (A) in Theorem \ref{weaklypositive},
there exists $x \in X_{ \mathcal{F} }  $ such that $E$ is weakly positive at $x$.
(We use Demailly's $L^2$ estimate on a complete \kah manifold in \cite[Theorem 5.1]{Dem82}
instead of the injectivity theorem in \cite{CDM} since $X_{ \mathcal{F} }  $ may not be a weakly 1-complete. See also \cite[Theorem 2.5.3]{PT}.) % or \cite[Lemma 2.4]{Tak}).
Hence $\mathcal{F}$ is pseudo-effective.

(2) The proof is similar to the proof of (1) and Theorem \ref{weakviehweg}.
(We can take $h_{k}$ such that $h_k$ is smooth on $U \cap X_{\mathcal{F}}$ 
if $\mathcal{F}$ is weakly positive on $U$.)

(3)($\Rightarrow$).
The proof is similar to the proof of (1).

(3)($\Leftarrow$).
By Corollary \ref{sHmweak} and (1), $ \widehat{\Sym}^{2k}( \mathcal{F}) \otimes A^{-1} $ is 
pseudo-effective.
%weakly positive
%at $x \in X_{\mathcal{F}}$.
%Therefore there exists $b \in \mathbb{N}_{>0}$ such that
% $\widehat{\Sym}^{2kb}( \mathcal{F}) \otimes A^{-b} =\widehat{\Sym}^{2kb}( \mathcal{F}) \otimes A^{-2b} \otimes A^{b} $ is globally generated at $x$, 
% which completes the proof.
\end{proof}

We give an application of Theorem \ref{torsionfreecoh}.
If a torsion-free coherent sheaf $\mathcal{F}$ has a Griffiths semipositive singular Hermitian metric $h$,
then $\mathcal{F}$ is pseudo-effective. %weakly positive in the sense of Nakayama.
Moreover if there exists a Zariski open set $U$ such that $h$ is continuous on $U \cap X_{\mathcal{F}}$, then $\mathcal{F}$ is weakly positive at any $x \in U \cap X_{\mathcal{F}}$.

Let $f \colon X \rightarrow Y $ be a surjective morphism between smooth projective varieties.
Then for any $m \in \mathbb{N}_{>0}$, $f_{*}( mK_{X/Y})$ has a Griffiths semipositive singular Hermitian metric, which
%$h_{mNS}$ such that$h_{mNS}$ 
is continuous over the regular locus of $f$ by \cite[Theorem 1.1]{PT} or \cite[Theorem 27.1]{HPS}.
Therefore, $f_{*}( mK_{X/Y})$ is weakly positive at any point of the regular locus of $f$.
This was already proved in \cite[Theorem 5.1.2]{PT}.

\section{An augmented base locus and a Griffiths semipositive singular Hermitian metric.}
\label{augmented}
In this chapter, we study a relationship between an augmented base locus and a Griffiths semipositive singular Hermitian metric of a vector bundle.
%この章ではベクトル束のaugumented base locusと特異エルミート計量の関係を調べる.

\subsection{An augmented base locus and a restiricted base locus}
\textit{}

We review some of the standard facts of augmented  base loci and restiricted base loci of vector bundles.
\begin{dfn} \cite[Section 2]{BKKMSU}
Let $X$ be a smooth projective variety and $E$ be a holomorphic vector bundle on $X$.
The {\it base locus} of $E$ is defined by
$$
{\rm Bs}(E) \coloneqq \{ x \in X \colon H^0(X , E ) \rightarrow E_x \mbox{ is not surjective} \},
$$
and the {\it stable base locus} of $E$ is defined by
$$
\mathbb{B}(E) \coloneqq \bigcap_{m > 0} {\rm Bs}(\Sym^m(E)).
 $$

Let $A$ be an ample line bundle.
We define
the {\it augmented base locus} of $E$ by
$$
\mathbb{B}_{+}(E) = \bigcap_{ m\in \mathbb{N}_{>0}} \mathbb{B}( \Sym^m(E) \otimes A^{-1})
$$
and the {\it restricted base locus} of $E$ by
$$
\mathbb{B}_{-}(E) = \bigcup_{m\in \mathbb{N}_{>0}} \mathbb{B}( \Sym^m(E) \otimes A).
$$
\end{dfn}

We point out that both $\mathbb{B}_{+}(E)$ and $\mathbb{B}_{-}(E) $ do not depend on the choice of the ample line bundle $A$ by \cite[Remark 2.7]{BKKMSU}.
%We will denote by $\pi \colon \mathbb{P}(E) \rightarrow X$ the projective bundle of rank one quotients of $E$  and by $ \mathcal{O}_{ \mathbb{P}(E) }(1) $ the universal quotient of $\pi^{*} E$ on $\mathbb{P}(E)$.
%We recall the definition of a pseudo-effective vector bundle.We adopt the definition in \cite{BDPP}.

\begin{thm}\cite[Proposition 3.1, 3.2, 7.2 and Theorem 6.4]{BKKMSU}
\label{aug_vect}
Let $X$ be a smooth projective variety and $E$ be a holomorphic vector bundle on $X$.
Then the followings hold.
\begin{enumerate}
\item $\pi({\mathbb{B}}_{+}(\mathcal{O}_{ \mathbb{P}(E) }(1)) ) = \mathbb{B}_{+}(E)$ and  $\pi({\mathbb{B}}_{-}(\mathcal{O}_{ \mathbb{P}(E) }(1)) ) = \mathbb{B}_{-}(E)$.
\item $\mathbb{B}_{+}(E) \neq X $ iff $E$ is big.
\item $\mathbb{B}_{-}(E) \neq X$ iff $E$ is pseudo-effective.
\item $\overline{\mathbb{B}_{-}(E)} \neq X$ iff $E$ is weakly positive, where $\overline{\mathbb{B}_{-}(E)} $ is a Zariski closure of $\mathbb{B}_{-}(E)$ in $X$.
\end{enumerate}

\end{thm}
Let $X$ be a smooth projective variety, $L$ be a line bundle, and $h$ be a singular Hermitian metric with semipositive curvature current on $L$.
Let $P_{h}$ denote the set of $x \in X$ such that $\nu(h,x) >0$.
By \cite[Th\'eor\`eme 2.1.20 and Corollaire 2.2.8]{Bou02}, 
if $L$ is big, then $\mathbb{B}_{+}(L) = P_{h_{min}}$, where $h_{min}$ is a minimal singular Hermitian metric of $L$ (see \cite[Theorem 1.5]{DPS}).

\begin{cor}
\label{finsler}
Let $X$ be a smooth projective variety, $E$ be a holomorphic vector bundle on $X$, $\omega$ be a \kah form on $ \mathbb{P}(E)$.
Then the followings hold.
\begin{enumerate}
\item $E$ is big iff 
$ \mathcal{O}_{ \mathbb{P}(E) }(1) $ admits a singular Hermitian metric $h$ such that $\pi(P_{h}) \neq X$ and $\sqrt{-1}\Theta_{\mathcal{O}_{ \mathbb{P}(E) }(1)  , h} \ge \epsilon \omega$, for some $\epsilon >0$.
\item $E$ is pseudo-effective iff there exists a singular Hermitian metric $h_k$ on $\mathcal{O}_{ \mathbb{P}(E) }(1)$ such that $\pi(P_{h_k}) \neq X$ and $\sqrt{-1}\Theta_{ \mathcal{O}_{ \mathbb{P}(E) }(1) , h_k}  \ge \frac{-\omega}{k}$ for any $k \in \mathbb{N}_{>0}$, 
\item $E$ is weakly positive iff there exists a proper Zariski closed set $Z \subset X$ such that $\mathcal{O}_{ \mathbb{P}(E) }(1)$ has a singular Hermitian metric $h_k$ with $\pi(P_{h_k}) \subset Z$ and $\sqrt{-1}\Theta_{ \mathcal{O}_{ \mathbb{P}(E) }(1) , h_k}  \ge \frac{-\omega}{k}$ for any $k \in \mathbb{N}_{>0}$, 
\end{enumerate}
\end{cor}

\begin{proof}
We may assume that $A$ is an ample line bundle on $\mathbb{P}(E) $ and $h_A$ is a smooth positive metric on $A$ such that $\sqrt{-1}\Theta_{A,h_A} = \omega$.
%$A$を$\mathbb{P}(E) $上のampleで$h_A$をsmooth positive metric on $A$で$\sqrt{-1}\Theta_{A,h_A} = \omega$と仮定して良い.

(1)($\Rightarrow$).
By  \cite[Remark 2.7]{BKKMSU}, 
there exists $m \in \mathbb{N}_{>0}$ such that  
$\mathbb{B}_{+}(E) = 
\pi(\mathbb{B}(\mathcal{O}_{ \mathbb{P}(E) }(m) \otimes A^{-1})$.
Hence $\mathcal{O}_{ \mathbb{P}(E) }(m) \otimes A^{-1}$ has a singular Hermitian metric $h_{m}$ with semipositive curvature current with $\pi(P_{h_m}) = \mathbb{B}_{+}(E)$.
Set $h:=(h_{m} h_A)^{\frac{1}{m}}$, which completes the proof.

%(1)($\Leftarrow$).By  \cite[Remark 2.7]{BKKMSU}, ある$m \in \mathbb{N}_{>0}$があって, $\mathbb{B}_{+}(E) = \pi(\mathbb{B}(\mathcal{O}_{ \mathbb{P}(E) }(m) \otimes A^{-1})$となる.よって$\mathcal{O}_{ \mathbb{P}(E) }(m) \otimes A^{-1}$には特異計量$h_{m}$でsemipositive curvature currentであり, $\pi(P_{h_m}) = \mathbb{B}_{+}(E) \neq X$となる.これより, $h:=(h_{m} h_A)^{\frac{1}{m}}$とおけば証明が完了する.

(1)($\Leftarrow$).
By the assumption, $\mathcal{O}_{ \mathbb{P}(E) }(1)$ is big.
By Theorem \ref{aug_vect} and \cite{Bou02}, 
$$
\mathbb{B}_{+}(E) = \pi(\mathbb{B}_{+}(\mathcal{O}_{ \mathbb{P}(E) }(1)) )= \pi(P_{h_{min}})  \subset \pi(P_{h}),
$$
where $h_{min}$ is a minimal singular Hermitian metric on $\mathcal{O}_{ \mathbb{P}(E) }(1)$.
Therefore $E$ is big.

%(1)($\Rightarrow$).仮定より, $\mathcal{O}_{ \mathbb{P}(E) }(1)$はbigとなる.By Theorem \ref{aug_vect} and \cite{Bou02}, $$\mathbb{B}_{+}(E) = \pi(\mathbb{B}_{+}(L) )= \pi(P_{h_{min}})  \subset \pi(P_{h})   \neq X$$, where $h_{min}$ はminimal singular Hermitian metric on $\mathcal{O}_{ \mathbb{P}(E) }(1)$.

(2)($\Rightarrow$).
Fix $k \in \mathbb{N}_{>0}$. 
By Theorem \ref{aug_vect},
$$\pi(\mathbb{B}_{+}(\mathcal{O}_{ \mathbb{P}(E) }(k) \otimes A) ) \subset 
\cup_{m \in \mathbb{N}_{>0} }\pi(\mathbb{B}_{+}(\mathcal{O}_{ \mathbb{P}(E) }(m) \otimes A) )
 = \mathbb{B}_{-}(E).$$
Let $h_{min,k}$ be a minimal singular metric on $\mathcal{O}_{ \mathbb{P}(E) }(k) \otimes A$.
Set $h_k:=(h_{min,k}h_{A}^{-1})^{\frac{1}{k}}$.
Then $\pi(P_{h_k}) \neq X$ and $\sqrt{-1}\Theta_{ \mathcal{O}_{ \mathbb{P}(E) }(1) , h_k}  \ge \frac{-\omega}{k}$.

%(2)($\Leftarrow$).$k \in \mathbb{N}_{>0}$をfixする. By Theorem \ref{aug_vect}$$\pi(B_{+}(\mathcal{O}_{ \mathbb{P}(E) }(k) \otimes A) ) \subset \cup_{m \in \mathbb{N}_{>0} }\pi(B_{+}(\mathcal{O}_{ \mathbb{P}(E) }(m) \otimes A) ) = \mathbb{B}_{-}(E) \neq X$$から, $h_{min,k}$を$\mathcal{O}_{ \mathbb{P}(E) }(k) \otimes A$のminimal singular metricとし$h_k:=(h_{min,k}h_{A}^{-1})^{\frac{1}{k}}$とおくと, $\pi(P_{h_k}) \neq X$ and $\sqrt{-1}\Theta_{ \mathcal{O}_{ \mathbb{P}(E) }(1) , h_k}  \ge \frac{-\omega}{k}$

(2)($\Leftarrow$).
Set $H_{k}:=h_{k}^{k} h_A$ for any $k \in \mathbb{N}_{>0}$.
Then $H_{k}$ is a singular Hermitian metric  with semipositve curvature current on $\mathcal{O}_{ \mathbb{P}(E) }(k) \otimes A$.
By Theorem \ref{aug_vect} and \cite{Bou02}, 
$$
\mathbb{B}_{-}(E) = 
\cup_{m \in \mathbb{N}_{>0} } \pi(\mathbb{B}_{+}(\mathcal{O}_{ \mathbb{P}(E) }(m) \otimes A) ))
\subset 
\cup_{m \in \mathbb{N}_{>0} }  \pi(P_{H_{m}}).
$$
Therefore $E$ is pseudo-effective.

(3).
The proof is same as that of (2).
%証明は(2)と同じである.
%Set $\overline{\mathbb{B}_{-}(E)}$. 
%証明は(2)($\Leftarrow$).と同じ
%(3)($\Rightarrow$).(2)($\Rightarrow$).と同じである.
\end{proof}
\begin{rem}
In the hypotheses of Corollary \ref{finsler}, 
if $\mathcal{O}_{ \mathbb{P}(E) }(1) $ has a singular Hermitian metric $h$ with semipositive curvature current such that $\pi(P_{h}) \neq X$, then $E$ is pseudo-effective by 
$\pi(\mathbb{B}_{-}(\mathcal{O}_{\mathbb{P}(E)}(1))) \subset \pi(P_{h}) $.
However, the converse does not necessarily hold. 

Let $X$ be an elliptic curve.
$E$ is defined by the nontrivial exact sequence of vector bundles:
$$
0 \rightarrow  \mathcal{O}_{X} \rightarrow E \rightarrow  \mathcal{O}_{X} \rightarrow 0.
$$
Then $E$ is nef.
Let $h_{min}$ be the minimal singular Hermitian metric on $\mathcal{O}_{ \mathbb{P}(E) }(1) $.
By \cite[Example 1.7]{DPS1}, $h_{min}$ has a pole and $\pi(P_{h_{min}}) = X$.

\end{rem}

%\section{Lelong number and non K\"{a}hler locus}

\subsection{A Lelong number of a singular Hermitian metric on a vector bundle}
\textit{}

The Lelong number of a singular Hermitian metric on a vector bundle was defined by Berndtsson \cite{Bernd2}.
Based on the Berndtsson's work, we define a new Lelong number.

Let $U$ be a unit ball in $\mathbb{C}^n$, and $E = U \times \mathbb{C}^r$,
$h$ be a Griffiths semipositive singular Hermitian metric.
We take a standard frame $e_1 , \cdots , e_r$ of $E$.
Then we have $ \mathbb{P}(E) = U \times \mathbb{P}^{r-1}$ and $\mathcal{O}_{ \mathbb{P}(E) }(1)$
can be endowed with a singular Hermitian metric $g $ with semipositive curvature current induced by $h$.
We have $g = e^{- \varphi (z , W)}$,
where $\varphi (z , W)$ is a quasi-plurisubharmonic function on  $U \times \mathbb{P}^{r-1}$.

\begin{dfn}
\label{Lelongnumberdfn}
We will denote by $\nu( \varphi , (0 , W))$  %:=  liminf_{(z ,U) \rightarrow (0 , W)}  \frac{ \varphi(z , U) }{ \log ( \| z \| +\| W - U\| )}$
the Lelong number $\varphi$ at $(0 , W) \in \mathbb{P}(E)$. %where $\| W \|^2 := \sum_{1 \le i \le r} |W_i|^2 $.
We define the following numbers:
$$
\nu_{\sup} ( h , 0) := \sup_{ W \in \mathbb{P}^{r-1}} \nu( \varphi , (0 , W)) ,
\,\,\,\,
\nu_{\inf} ( h , 0) := \inf_{ W \in \mathbb{P}^{r-1}} \nu( \varphi , (0 , W)).
$$
\end{dfn}

We explain more explicitly.
Let $h^{*} = ( h^{*}_{ij})$ be the dual metric of $h$ on $E$.
We take the chart $\{ [W_1 : \cdots : W_r ] \in \mathbb{P}^{r-1} \colon W_r \neq  0\}$ of $\mathbb{P}(E)$.
As in Lemma \ref{analytic}, we define the isomorphism by 
$$
\begin{array}{ccc}
U \times \{ W_r \neq  0\}   & \rightarrow & U \times \mathbb{C}^{r-1} \\
(z , [W_1 : \cdots : W_r])  & \rightarrow  & (z , \frac{W_1}{W_r} , \cdots , \frac{W_{r-1}}{W_r}).
\end{array}
$$
We may regard $U  \times \{ W_r \neq  0\} $  as $ U  \times \mathbb{C}^{r-1} $.
Put $\eta_{l} := \frac{W_{l}}{W_r} $ for $1 \le l \le r-1$ and $\eta_{r} := 1$.
In this setting, we have
$$
 \mathcal{O}_{ \mathbb{P}(E) }(-1) |_{ U  \times \mathbb{C}^{r-1} } = 
 \{ (z ,  \eta , \xi) \in U \times \mathbb{C}^{r-1} \times \mathbb{C}^r
 \colon
 \eta_{i} \xi_{j} = \eta_{j} \xi_{i} \}
$$
and the local section 
$$
e_{ \mathcal{O}_{ \mathbb{P}(E) }(-1) } ( z , (\eta_1 , \cdots, \eta_{r-1 } ))
= 
(z , (\eta_1 , \cdots , \eta_{r-1 } ) ,(\eta_1 , \cdots , \eta_{r-1 }  , 1) ).
$$
Then the dual metric $g^{*} = e ^{ \varphi}$ on $\mathcal{O}_{ \mathbb{P}(E) }(-1)$ is written by
$$
g^{*}(z , \eta)  = | (\eta_1 , \cdots \eta_{r-1 }  , 1)  |^{2}_{ h^{*}} = \sum_{1 \le i , j \le r }h^{*}_{ij} (z)\eta_i \bar{\eta_j}     
= \frac{ \sum_{1 \le i , j \le r} h^{*}_{ij} (z)W_i \bar{W_j}     }{ |W_r |^2}.$$
Therefore, for any $W = [W_1 : \cdots : W_r ]  \in \{ W_r \neq  0\} $, we have 
$$\varphi (z , W)= \log ( \sum_{1 \le i , j \le r} h^{*}_{ij} (z)W_i \bar{W_j} ) - 2 \log  |W_r |$$
and 
$$\nu( \varphi , (0 , W)) = \nu (  \log ( \sum_{1 \le i , j \le r} h^{*}_{ij} (z)W_i \bar{W_j} ) , (0 , W) ).$$
%%%%%%%%%%%%%%%%%%%%%%%%
\begin{comment}
We have
$$
 \mathcal{O}_{ \mathbb{P}(E) }(-1) = 
 \{ (z , W , \eta ) \in U \times \mathbb{P}^{r-1} \times \mathbb{C}^r
 \colon
 W_i \eta_{j} = W_j \eta_i \}
$$
and the standard inclusion
$ \mathcal{O}_{ \mathbb{P}(E) }(-1) \rightarrow \pi^{*} E^{\vee} = U \times  \mathbb{P}^{r-1} \times \mathbb{C}^r $.
As in Lemma \ref{analytic}, the local frame $\widetilde{e}$ on $U \times \mathbb{P}^{r-1}$ of $\mathcal{O}_{ \mathbb{P}(E) }(-1)$
is defined by 
$$
\widetilde{e} (z , W) := (z , W , \frac{ W }{ \| W \|}),
$$
where $\| W \| := (\sum_{1 \le i \le r} |W_i|^2)^{\frac{1}{2} }$.
Then the dual metric $g^{*} = e ^{ \varphi}$ on $\mathcal{O}_{ \mathbb{P}(E) }(-1)$ is written by
$$
g^{*}(z , W )  = | \frac{ W }{ \| W \|} |^{2}_{ h^{*}} = \frac{ \sum_{1 \le i , j \le r} h^{*}_{ij} (z)W_i \bar{W_j}     }{\|W\|^2}.
$$
Therefore we have 
$$\varphi (z , W)= \log ( \sum_{1 \le i , j \le r} h^{*}_{ij} (z)W_i \bar{W_j} ) - 2 \log \|W \|$$
and 
$$\nu( \varphi , (0 , W)) = \nu (  \log ( \sum_{1 \le i , j \le r} h^{*}_{ij} (z)W_i \bar{W_j} ) , (0 , W) ).$$
\end{comment}
%%%%%%%%%%%%%%%%%%%%%%%%%%%%%%%

In this setting, we explain the relationship with the Lelong number defined by Berndtsson \cite{Bernd2}.
For any $a = (a_1 , \cdots , a_r) \in \mathbb{C}^{r} \setminus \{ 0 \}$, 
we write $u_{a} = \sum_{1 \le i \le r} a_{i} e_{i} ^{*}$.
In \cite{Bernd2}, the Lelong number of $h^{*}$  at 0 in the direction $u_a$ is defined by 
$$
\gamma_{h^{*} }( u_a , 0) = \liminf_{z \rightarrow 0} \frac{ \log |u_a|^{2}_{h^{*} } } {\log | z|}.
$$

From  $\log |u_a|^{2}_{h^{*} }  =  \log (\sum_{1 \le i , j \le r} h^{*}_{ij} (z)a_i \bar{a_j})$,
we have the following corollary.
\begin{cor}
In the above setting, the following hold.
\begin{enumerate}
\item $\gamma_{h^{*} }( u_a , 0) = \nu( \varphi |_{ U \times \{ [a_1 :  \cdots  : a_r ]\} }, 0)
\ge \nu( \varphi , (0 ,  [a_1 :  \cdots  : a_r ]) )$
 for any  $a = (a_1 , \cdots , a_r) \in \mathbb{C}^{r} \setminus \{ 0 \}$.
\item  $\gamma_{h^{*} }( u_a , 0) = \nu( \varphi , (0 ,  [a_1 :  \cdots  : a_r ]) )$ 
for general $a = (a_1 , \cdots , a_r) \in \mathbb{C}^{r} \setminus \{ 0 \}$.
\item $\nu_{\inf} ( h , 0) = \inf_{ a \in \mathbb{C}^r \setminus \{ 0 \} }\gamma_{h^{*} }( u_a , 0)$.%, where $u = \sum_{1 \le i \le r} a_{i} e_{i} ^{*}$.
\end{enumerate}
\end{cor}
\begin{proof}
The first equality of (1) is clear.
By \cite[Theorem 7.13]{Dem2} we have the second inequality of (1) and the equality of (2).

We now prove (3).
By (1), we have $ \nu_{\inf} ( h , 0) \le \inf_{ a \in  \mathbb{C}^r \setminus \{ 0 \} }\gamma_{h^{*} }( u_a , 0)$.
By \cite[Lemma 2.17]{Dem}, we have 
$$
\nu_{\inf} ( h , 0) = \nu( \varphi , \{ 0 \} \times \mathbb{P}^{r-1} ) = \nu( \varphi , (0 , [a_1 :  \cdots  : a_r ] )) 
$$
for general $ (a_1 , \cdots , a_r) \in \mathbb{C}^{r} \setminus \{ 0 \}$.
Therefore by using (2), the proof is complete.
\end{proof}

%We define the following map.
%$$
%G \colon U \times \mathbb{P}^{r-1} \rightarrow \mathbb{R}
%(z , [ W_1 : \cdots W_r]) \rightarrow \frac{ \| \sum_{1 \le i \le r}W_i e_{i}^{*} \|^{2} _{h^{*}} }{ \| W \|^2 } ,
%$$
%where $\| W \|^2 := \sum_{1 \le i \le r} |W_i|^2 $

%\begin{dfn}
%Let $h$ be a semipositive singular Hermitian metric on 
%\end{dfn}

\subsection{A higher rank analogy of Boucksom's non-\kah locus}
\textit{}

%\section{The relationship of the divisorial Zariski decomposition}
%\section{On the case of compact \kah manifold}

We introduce the non-\kah loci %and non-nef locus 
on vector bundles.

\begin{dfn}
Let $X$ be a smooth projective variety and $E$ be a holomorphic vector bundle on $X$.
\begin{enumerate}
\item For any $k \in \mathbb{N}_{>0}$ and any ample line bundle $A$, we set
$$
\mathcal{H}_{k , A} ^{+}= \{ h \colon h \mbox{ is a Griffiths semipositive singular Hermitian metric of } 
\Sym^{k}(E) \otimes A^{-1} \}
$$ 
%$$
%\mathcal{H}_{k , A} ^{-}= \{ h \colon h \mbox{ is a Griffiths semipositive singular Hermitian metric of } S^{k} \otimes A \}.
%$$
\item If $E$ is big, the {\it non-\kah locus} $L_{nK}( E)$ is defined by
$$
L_{nK}( E) := \bigcap_{ k , A} \bigcap_{h \in \mathcal{H}_{k , A} ^{+}}\{ x \in X \colon \nu_{\sup} ( h , x) > 0\},
$$
where the cap is taken over all $k \in \mathbb{N}_{>0}$ and all ample line bundle $A$ on $X$.
By Theorem \ref{maintheorem}, this locus is well-defined.
%\item If $E$ is pseudo-effective, the non-nef locus $L_{nn}(E)$ is defined by
%$$
%L_{nK}( E) := \bigcup_{ k , A}  \{ x \in X \colon \inf_{ h \in  \mathcal{H}_{k , A} ^{-}    }(  \nu_{\sup} ( h , x) ) > 0\},
%$$
%where the cup is taken over all positive number $k$ and all ample line bundle $A$.
%By Theorem \ref{maintheorem}, this locus is well-defined.
\end{enumerate}
\end{dfn}

This is a higher rank analogy of Boucksom's %non-nef locus and 
non-\kah locus in \cite{Bou}.
In this section, we prove the following theorem.

\begin{thm}
\label{Boucksom}
If $E$ is big, then $L_{nK}( E) =\mathbb{B}_{+}(E)$.
%If $E$ is pseudo-effective, $L_{nK}( E) = \mathbb{B}_{-}(E)$ holds.
\end{thm}

Therefore, we give a characterization of  the augmented base loci %and restiricted base locus 
by using singular Hermitian metrics on vector bundles and the Lelong numbers as in \cite[Corollaire 2.2.8]{Bou02}.

Before the proof, we recall  a singular Hermitian metric induced by holomorphic sections, proposed by Hosono \cite[Chapter 4]{Hos17}.
%such that $E_{y}$is generated by $s_1(y), \dots, s_N(y)$ for a general point $y$.
We assume that $E$ is a globally generated at a general point.
%For any point $x \in X$, 
Let $s_1, \dots, s_N \in H^0( X , E )$ be holomorphic sections. 
We take a local coordinate $U$ %$( U ; z_1 , \dots , z_n)$near $x$ and 
and take a local holomorphic frame $e_1, \dots, e_r$ of $E$ on $U$.
Write $s_{\alpha }= \sum_{1 \le j \le r} f_{ \alpha j}e_{j} $, where $f_{ \alpha j}$ are holomorphic functions on $U$.
A singular Hermitian metric $h_s$ induced by $s_1, \dots, s_N$ is given by
$$
(h^{*}_{s} )_{jk } := \sum_{1 \le \alpha \le N} f_{\alpha j } \bar{ f_{\alpha k} } .
$$
By \cite[Example 3.6 and Proposition 4.1]{Hos17}, $h_s$ is Griffiths semipositive.%positively curved and $E(h)$ is a coherent sheaf.
\begin{prop}
\label{Hosono}
In this setting, 
 $\{ x \in X \colon \nu_{\sup} ( h_s , x) >0 \} \subset Bs( E)$.
\end{prop}

\begin{proof}
The $N \times r$ matrix $A$ is defined by
$
A_{\alpha j} = f_{ \alpha j}
$
as in Lemma \ref{analytic}.
By the standard linear algebra, we have $Bs(E) \cap U = \{ x \in U \colon \text{rank } A(x) < r \}$.

Let $g = e^{-\varphi}$ be a singular Hermitian metric with semipositive curvature current on $\mathcal{O}_{ \mathbb{P}(E) }(1)$
induced by $h_s$.
%By the argument of Section 6, we have
By the above argument, we have
$$
\nu( \varphi , (z , W)) = \nu (  \log (  \sum_{1 \le j,k\le r , 1 \le \alpha \le N} f_{\alpha j } W_j \overline{ f_{\alpha k} W_k}) , (0 , W) ).
$$
%$$
%\varphi (z , W)= \log ( \sum_{1 \le j , k \le r}(h^{*}_{s} )_{jk} (z)W_j \bar{W_k} ) - 2 \log \|W \|
%= \log ( \sum_{1 \le j,k\le r , 1 \le \alpha \le N} f_{\alpha j } W_j \overline{ f_{\alpha k} W_k} ) - 2 \log \|W \|.
%$$
If $ \nu_{\sup} ( h_s , x) >0 $, there exists $a \in \mathbb{P}^{r-1}$ such that 
$\nu ( \varphi  , ( x , a ) ) >0$.
We obtain 
$$ \sum_{1 \le j,k\le r , 1 \le \alpha \le N} f_{\alpha j }(x) a_j \overline{ f_{\alpha k}(x) a_k} = 0,
$$
and consequently $\sum_{1 \le j \le r}f_{\alpha j }(x) a_j  = 0$ for any $1 \le \alpha \le N$. 
Hence $\text{rank } A(x) < r$, therefore  $x \in Bs( E)$.
\end{proof}

Now, we prove the Theorem \ref{Boucksom}.
\begin{proof}
First we show that $L_{nK}( E) \subset \mathbb{B}_{+}(E)$.  We take a sufficiently ample line bundle $A$ such that 
$ \mathbb{B}_{+}(E) = \bigcap_{  q \in \mathbb{N}_{>0}} \mathbb{B}( \Sym^q(E) \otimes A^{-1})$ by \cite[Remark 2.7]{BKKMSU}.
It is enough to show that  $L_{nK}( E) \subset \mathbb{B}( \Sym^q(E) \otimes A^{-1})$ for any $q \in \mathbb{N}_{>0}$ such that $\mathbb{B}( \Sym^q(E) \otimes A^{-1}) \neq X$.
Fix $q \in \mathbb{N}_{>0}$. We take $m \in \mathbb{N}_{>0} $ such that $\mathbb{B}( \Sym^q(E) \otimes A^{-1}) = Bs (\Sym^{qm}(E) \otimes A^{-m})$.
From $Bs (\Sym^{qm}(E) \otimes A^{-m}) \neq X$, $\Sym^{qm}(E) \otimes A^{-m}$ can be endowed with 
a Griffiths semipositive singular Hermitian metric $h$ induced by global sections and  
$\{ x \in X \colon \nu_{\sup} ( h , x) >0 \} \subset Bs (\Sym^{qm}(E) \otimes A^{-m})$ holds.
By $h \in \mathcal{H}_{qm , A^m} ^{+}$ and the definition of $L_{nK}( E)$, the proof is complete.

For inverse inclusion, we take $x \notin L_{nK}( E)$.
There exist $k \in \mathbb{N}_{>0}$, an ample line bundle $A$, and a Griffiths semipositive singular Hermitian metric $h$ on 
$\Sym^{k}(E) \otimes A^{-1}$ such that $\nu_{\sup} ( h , x) =0$.
We will show that $x \notin \mathbb{B}_{+}(E)$, more precisely, there exists $q \in \mathbb{N}_{>0}$ such that
 $\Sym^{q}(E) \otimes A^{-1}$ is globally generated at $x$.
This proof is similar to the proof of Theorem \ref{weaklypositive}.

We take a local coordinate $(U ; z_1 , \cdots , z_n)$ near $x$.
Let $ \phi = \eta(n+1) \log |z-x|^2$, where $\eta $ is a cut-off function such that $\eta \equiv 1 $ near $x$.
Put $\psi := \frac{n}{n+1} \pi^{*} \phi$.
Let $h_{A}$ be a smooth positive Hermitian metric on $A$.
We take $m \in \mathbb{N}_{>0}$ such that 
\begin{enumerate}
\item $ m\sqrt{-1} \Theta_{A, h_{A}} + \deldel \eta \ge 0$ in the sense of current, and
\item $\mathcal{O}_{ \mathbb{P}(E) }(r) \otimes \pi^{*}( A^{-1} \otimes K_{X}^{-1} \otimes \det E^{*} \otimes A^m)$ has a smooth positive metric $h_0$.
\end{enumerate}

We will denote by $g$ the singular Hermitian metric with semipositive curvature current on $\mathcal{O}_{ \mathbb{P}(E) }(k) \otimes  A^{-1}$ induced by $h$.
From  $\nu_{\sup} ( h , x) =0$, there exists a open set 
$x \in V \subset \subset U $ such that $g^{2m}$ is integrable on $\pi^{-1}(V)$ by Skoda's Theorem. %and the definition of $\nu_{\sup}$.
Set 
$$\widetilde{ L} := (\mathcal{O}_{ \mathbb{P}(E) }(2km) \otimes  \pi^{*}A^{-2m} )\otimes
(\mathcal{O}_{ \mathbb{P}(E) }(r) \otimes \pi^{*}( A^{-1} \otimes K_{X}^{-1} \otimes \det E^{*} \otimes A^m)) \otimes  \pi^{*}A^{m}$$
and $\widetilde{h}:=g^{2m}h_0\pi^{*}(h_{A}^{m})$.
Then we have
\begin{align*}
\mathcal{O}_{ \mathbb{P}(E) }(2km) \otimes  \pi^{*}A^{-1}
&=
K_{\mathbb{P}(E)} \otimes \mathcal{O}_{ \mathbb{P}(E) }(2km) \otimes  \pi^{*}A^{-2m} \otimes
\mathcal{O}_{ \mathbb{P}(E) }(r) \otimes \pi^{*}( A^{-1} \otimes K_{X}^{-1} \otimes \det E^{*} \otimes A^m)  \otimes  \pi^{*}A^{m }\\
&= K_{\mathbb{P}(E)} \otimes \widetilde{ L}
\end{align*}

By the same argument of Theorem \ref{weaklypositive}, $\widetilde{ L}$ has a singular Hermitian 
metric $\widetilde{h}$ with semipositive curvature current such that
 $$
\sqrt{-1} \Theta_{ \widetilde{ L} ,    \widetilde{h}  } + ( 1 + \frac{\alpha}{n})\deldel \psi \ge 0
\mbox{  in the sense of current }
$$
for any $ \alpha \in [0 , 1]$.
In this setting, 
$\Sym^{2km}(E) \otimes A^{-1}$ is globally generated at $x$, since
the same proof as in Step 3 of (C) $\Rightarrow$ (A) in Theorem \ref{weaklypositive} works.
%The details left to the reader.
%(2) This proof is similar to (1).
%First, we show that $L_{nn}( E) \subset \mathbb{B}_{-}(E)$
%We take a sufficiently ample line bundle $A$ such that 
%$ \mathbb{B}_{-}(E) = \bigcup_{  q \in \mathbb{N}_{>0}} \mathbb{B}( S^q(E) \otimes A)$
\end{proof}
It would be interesting to be able to describe non nef locus by using singular Hermitian metrics as in Theorem \ref{Boucksom}.
\end{document}